\documentclass{siamonline181217}
\usepackage{algorithm}
\usepackage[noend]{algpseudocode}
\usepackage{multirow}
\usepackage{varwidth}
\usepackage{xy}
\usepackage{amsfonts}
\usepackage{ stmaryrd }
\usepackage{algorithmicx}
\usepackage{algorithm}
\usepackage{algpseudocode}
\usepackage{algorithm}
\usepackage[noend]{algpseudocode}
\usepackage[mathscr]{euscript}
\usepackage{graphicx}
\xyoption{all}
\newsiamremark{remark}{Remark}
\newsiamremark{example}{Example}

\usepackage{enumitem}
\setlist[enumerate]{leftmargin=.5in}
\setlist[itemize]{leftmargin=.5in}

\newcommand{\e}{{\rm e}}
\newcommand{\OO}{{\rm O}}
\newcommand{\R}{{\mathbb R}}
\newcommand{\EE}{{\mathbb E}}
\newcommand{\tr}{\operatorname{trace}}

\newcommand{\Idiagram}{\xymatrix@=2.9ex{
*{\bullet} \ar@[red]@<-.2ex>@{-}[d]\ar@{-}@[green][d] \ar@{-}@[blue]@<.2ex>[d] \\
*{\bullet}}}
\newcommand{\IIdiagram}{\xymatrix@=2.9ex{
*{\bullet} \ar@[red]@<-.2ex>@{-}[d]\ar@{-}@[green][d] \ar@{-}@[blue]@<.2ex>[d]  & *{\bullet}\\
*{\bullet} & *{\bullet} \ar@[red]@<-.2ex>@{-}[u]\ar@{-}@[green][u] \ar@{-}@[blue]@<.2ex>[u]}} 
\newcommand{\OdiagramR}{\xymatrix@=2.9ex{
*{\bullet} \ar@[red]@{-}[r]\ar@{-}@<-.1ex>@[green][d] \ar@{-}@[blue]@<.1ex>[d]  & *{\bullet}\\
*{\bullet} & *{\bullet} \ar@[red]@{-}[l]\ar@{-}@<-.1ex>@[green][u] \ar@{-}@[blue]@<.1ex>[u]  }}
\newcommand{\OdiagramG}{\xymatrix@=2.9ex{
*{\bullet} \ar@[green]@{-}[r]\ar@{-}@<-.1ex>@[red][d] \ar@{-}@[blue]@<.1ex>[d]  & *{\bullet}\\
*{\bullet} & *{\bullet} \ar@[green]@{-}[l]\ar@{-}@<-.1ex>@[red][u] \ar@{-}@[blue]@<.1ex>[u]  }}
\newcommand{\OdiagramB}{\xymatrix@=2.9ex{
*{\bullet} \ar@[blue]@{-}[r]\ar@{-}@<-.1ex>@[red][d] \ar@{-}@[green]@<.1ex>[d]  & *{\bullet}\\
*{\bullet} & *{\bullet} \ar@[blue]@{-}[l]\ar@{-}@<-.1ex>@[red][u] \ar@{-}@[green]@<.1ex>[u]  }}
\newcommand{\Xdiagram}{\xymatrix@=2.9ex{
*{\bullet} \ar@[red]@{-}[r]\ar@{-}@[green][rd]  \ar@{-}@[blue][d]  & *{\bullet}\ar@{-}@[green][ld]   \\
*{\bullet} & *{\bullet} \ar@[red]@{-}[l] \ar@{-}@[blue][u] }}

\newcommand{\IIblack}{\vcenter{\xymatrix@=2.9ex{
*{\bullet}\ar@{-}[d] & *{\bullet}\ar@{-}[d]\\
*{\bullet} & *{\bullet}}}}
\newcommand{\Xblack}{\vcenter{\xymatrix@=2.9ex{
*{\bullet}\ar@{-}[rd] & *{\bullet}\ar@{-}[ld]\\
*{\bullet} & *{\bullet}}}}
\newcommand{\Zblack}{\vcenter{\xymatrix@=2.9ex{
*{\bullet}\ar@{-}[r] & *{\bullet} \\
*{\bullet}\ar@{-}[r]  & *{\bullet}}}}

\newcommand{\IIred}{\vcenter{\xymatrix@=2.9ex{
*{\bullet}\ar@[red]@{-}[d] & *{\bullet}\ar@[red]@{-}[d]\\
*{\bullet} & *{\bullet}}}}
\newcommand{\Xred}{\vcenter{\xymatrix@=2.9ex{
*{\bullet}\ar@[red]@{-}[rd] & *{\bullet}\ar@[red]@{-}[ld]\\
*{\bullet} & *{\bullet}}}}
\newcommand{\Zred}{\vcenter{\xymatrix@=2.9ex{
*{\bullet}\ar@[red]@{-}[r] & *{\bullet} \\
*{\bullet}\ar@[red]@{-}[r]  & *{\bullet}}}}

\newcommand{\IIgreen}{\vcenter{\xymatrix@=2.9ex{
*{\bullet}\ar@[green]@{-}[d] & *{\bullet}\ar@[green]@{-}[d]\\
*{\bullet} & *{\bullet}}}}
\newcommand{\Xgreen}{\vcenter{\xymatrix@=2.9ex{
*{\bullet}\ar@[green]@{-}[rd] & *{\bullet}\ar@[green]@{-}[ld]\\
*{\bullet} & *{\bullet}}}}
\newcommand{\Zgreen}{\vcenter{\xymatrix@=2.9ex{
*{\bullet}\ar@[green]@{-}[r] & *{\bullet} \\
*{\bullet}\ar@[green]@{-}[r]  & *{\bullet}}}}

\newcommand{\IIblue}{\vcenter{\xymatrix@=2.9ex{
*{\bullet}\ar@[blue]@{-}[d] & *{\bullet}\ar@[blue]@{-}[d]\\
*{\bullet} & *{\bullet}}}}
\newcommand{\Xblue}{\vcenter{\xymatrix@=2.9ex{
*{\bullet}\ar@[blue]@{-}[rd] & *{\bullet}\ar@[blue]@{-}[ld]\\
*{\bullet} & *{\bullet}}}}
\newcommand{\Zblue}{\vcenter{\xymatrix@=2.9ex{
*{\bullet}\ar@[blue]@{-}[r] & *{\bullet} \\
*{\bullet}\ar@[blue]@{-}[r]  & *{\bullet}}}}

\newcommand{\RedMouse}{\vcenter{\xymatrix@C=.3cm@R=.2cm{
1&& 1 \\
&*{\bullet}\ar@[red]@{-}[dd] \ar@[green]@{-}[lu] \ar@[blue]@{-}[ru]   &\\
&& \\
&*{\bullet}\ar@[green]@{-}[ld]  \ar@[blue]@{-}[rd]  &\\
2& & 2}}}

\newcommand{\GreenMouse}{\vcenter{\xymatrix@C=.3cm@R=.2cm{
1&& 1 \\
&*{\bullet}\ar@[green]@{-}[dd] \ar@[red]@{-}[lu] \ar@[blue]@{-}[ru]   &\\
 && \\
&*{\bullet}\ar@[red]@{-}[ld]  \ar@[blue]@{-}[rd]  &\\
2& & 2}}}

\newcommand{\BlueMouse}{\vcenter{\xymatrix@C=.3cm@R=.2cm{
1&& 1 \\
&*{\bullet}\ar@[blue]@{-}[dd] \ar@[red]@{-}[lu] \ar@[green]@{-}[ru]   &\\
& &\\
&*{\bullet}\ar@[red]@{-}[ld]  \ar@[green]@{-}[rd]  &\\
2& & 2}}}

\newcommand{\RedWorm}{\vcenter{\xymatrix@R=.5cm{
1\\
*{\bullet}\ar@[red]@{-}[u] \ar@[green]@<.15ex>@{-}[d]  \ar@[blue]@<-.15ex>@{-}[d]\\
*{\bullet}\ar@[red]@{-}[d]\\
2
}}}

\newcommand{\GreenWorm}{\vcenter{\xymatrix@R=.5cm{
1\\
*{\bullet}\ar@[green]@{-}[u] \ar@[red]@<.15ex>@{-}[d]  \ar@[blue]@<-.15ex>@{-}[d]\\
*{\bullet}\ar@[green]@{-}[d]\\
2
}}}

\newcommand{\BlueWorm}{\vcenter{\xymatrix@R=.5cm{
1\\
*{\bullet}\ar@[blue]@{-}[u] \ar@[red]@<.15ex>@{-}[d]  \ar@[green]@<-.15ex>@{-}[d]\\
*{\bullet}\ar@[blue]@{-}[d]\\
2
}}}

\newcommand{\Plankton}{\vcenter{\xymatrix@=.6cm{
*{\bullet}\ar@[red]@<.3ex>@{-}[d] \ar@[green]@{-}[d]  \ar@[blue]@<-.3ex>@{-}[d]\\
*{\bullet}}}}

\newcommand{\RedSnail}{\vcenter{\xymatrix@C=.2cm@R=.13cm{
1&  & 1  \\
& *{\bullet}  \ar@[green]@{-}[lu] \ar@[blue]@{-}[ru]  &  \\
&&\\
&*{\bullet}\ar@[red]@{-}[uu] \ar@[green]@<.15ex>@{-}[dd]  \ar@[blue]@<-.15ex>@{-}[dd] & \\
& &\\
&*{\bullet}\ar@[red]@{-}[dd] &\\
&&\\
&1 &
}}}

\newcommand{\GreenSnail}{\vcenter{\xymatrix@C=.2cm@R=.13cm{
1&  & 1  \\
& *{\bullet}  \ar@[red]@{-}[lu] \ar@[blue]@{-}[ru]  &  \\
&&\\
&*{\bullet}\ar@[green]@{-}[uu] \ar@[red]@<.15ex>@{-}[dd]  \ar@[blue]@<-.15ex>@{-}[dd] & \\
&&\\
&*{\bullet}\ar@[green]@{-}[dd] &\\
&&\\
&1 &
}}}

\newcommand{\BlueSnail}{\vcenter{\xymatrix@C=.2cm@R=.13cm{
1&  & 1  \\
& *{\bullet}  \ar@[red]@{-}[lu] \ar@[green]@{-}[ru]  &  \\
&&\\
&*{\bullet}\ar@[blue]@{-}[uu] \ar@[red]@<.15ex>@{-}[dd]  \ar@[green]@<-.15ex>@{-}[dd] & \\
&&\\
&*{\bullet}\ar@[blue]@{-}[dd] &\\
&&\\
&1 &
}}}

\newcommand{\Triangle}{\vcenter{\xymatrix@C=.17cm@R=.3cm{
&& 1 & &\\
&&  *{\bullet}  \ar@[red]@{-}[u]  \ar@[blue]@{-}[ldd]  \ar@[green]@{-}[rdd]&&\\
&&& &\\
&*{\bullet}   \ar@[green]@{-}[ld]  \ar@[red]@{-}[rr] & & *{\bullet}  \ar@[blue]@{-}[rd] &\\
1 & & && 1
}}}

\newcommand{\RedWeasel}{\vcenter{\xymatrix@C=.2cm@R=.13cm{
1&  & 1  \\
& *{\bullet}  \ar@[green]@{-}[lu] \ar@[blue]@{-}[ru]  &  \\
&&\\
&*{\bullet}\ar@[red]@{-}[uu] \ar@[green]@<.15ex>@{-}[dd]  \ar@[blue]@<-.15ex>@{-}[dd] & \\
& &\\
&*{\bullet}\ar@[red]@{-}[dd] &\\
&&\\
& *{\bullet}  \ar@[green]@{-}[ld] \ar@[blue]@{-}[rd]  &  \\
2 & & 2
}}}

\newcommand{\GreenWeasel}{\vcenter{\xymatrix@C=.2cm@R=.13cm{
1&  & 1  \\
& *{\bullet}  \ar@[red]@{-}[lu] \ar@[blue]@{-}[ru]  &  \\
&&\\
&*{\bullet}\ar@[green]@{-}[uu] \ar@[red]@<.15ex>@{-}[dd]  \ar@[blue]@<-.15ex>@{-}[dd] & \\
&&\\
&*{\bullet}\ar@[green]@{-}[dd] &\\
&&\\
& *{\bullet}  \ar@[red]@{-}[ld] \ar@[blue]@{-}[rd]  &  \\
2 & & 2
}}}

\newcommand{\BlueWeasel}{\vcenter{\xymatrix@C=.2cm@R=.13cm{
1&  & 1  \\
& *{\bullet}  \ar@[red]@{-}[lu] \ar@[green]@{-}[ru]  &  \\
&&\\
&*{\bullet}\ar@[blue]@{-}[uu] \ar@[red]@<.15ex>@{-}[dd]  \ar@[green]@<-.15ex>@{-}[dd] & \\
&&\\
&*{\bullet}\ar@[blue]@{-}[dd] &\\
&&\\
& *{\bullet}  \ar@[red]@{-}[ld] \ar@[green]@{-}[rd]  &  \\
2 & & 2
}}}

\newcommand{\RedSquare}{\vcenter{\xymatrix@C=.3cm@R=.3cm{
1& & & & 2 \\
&*{\bullet}\ar@[blue]@{-}[dd] \ar@[red]@{-}[lu] \ar@[green]@{-}[rr] & &*{\bullet} \ar@[red]@{-}[ru]   &\\
& & & & \\
&*{\bullet}\ar@[red]@{-}[ld] & &  *{\bullet} \ar@[red]@{-}[rd] \ar@[blue]@{-}[uu]  \ar@[green]@{-}[ll]  &\\
3& & & & 4}}}

\newcommand{\GreenSquare}{\vcenter{\xymatrix@C=.3cm@R=.3cm{
1& & & & 2 \\
&*{\bullet}\ar@[blue]@{-}[dd] \ar@[green]@{-}[lu] \ar@[red]@{-}[rr] & &*{\bullet} \ar@[green]@{-}[ru]   &\\
& & & & \\
&*{\bullet}\ar@[green]@{-}[ld] & &  *{\bullet} \ar@[green]@{-}[rd] \ar@[blue]@{-}[uu]  \ar@[red]@{-}[ll]  &\\
3& & & & 4}}}

\newcommand{\BlueSquare}{\vcenter{\xymatrix@C=.3cm@R=.3cm{
1& & & & 2 \\
&*{\bullet}\ar@[green]@{-}[dd] \ar@[blue]@{-}[lu] \ar@[red]@{-}[rr] & &*{\bullet} \ar@[blue]@{-}[ru]   &\\
& & & & \\
&*{\bullet}\ar@[blue]@{-}[ld] & &  *{\bullet} \ar@[blue]@{-}[rd] \ar@[green]@{-}[uu]  \ar@[red]@{-}[ll]  &\\
3& & & & 4}}}

\newcommand{\RedGreenSquare}{\vcenter{\xymatrix@C=.3cm@R=.3cm{
1& & & & 2 \\
&*{\bullet}\ar@[blue]@{-}[dd] \ar@[red]@{-}[lu] \ar@[green]@{-}[rr] & &*{\bullet} \ar@[red]@{-}[ru]   &\\
& & & & \\
&*{\bullet}\ar@[green]@{-}[ld] & &  *{\bullet} \ar@[green]@{-}[rd] \ar@[blue]@{-}[uu]  \ar@[red]@{-}[ll]  &\\
1& & & & 2}}}

\newcommand{\RedBlueSquare}{\vcenter{\xymatrix@C=.3cm@R=.3cm{
1& & & & 2 \\
&*{\bullet}\ar@[green]@{-}[dd] \ar@[red]@{-}[lu] \ar@[blue]@{-}[rr] & &*{\bullet} \ar@[red]@{-}[ru]   &\\
& & & & \\
&*{\bullet}\ar@[blue]@{-}[ld] & &  *{\bullet} \ar@[blue]@{-}[rd] \ar@[green]@{-}[uu]  \ar@[red]@{-}[ll]  &\\
1& & & & 2}}}

\newcommand{\GreenBlueSquare}{\vcenter{\xymatrix@C=.3cm@R=.3cm{
1& & & & 2 \\
&*{\bullet}\ar@[red]@{-}[dd] \ar@[green]@{-}[lu] \ar@[blue]@{-}[rr] & &*{\bullet} \ar@[green]@{-}[ru]   &\\
& & & & \\
&*{\bullet}\ar@[blue]@{-}[ld] & &  *{\bullet} \ar@[blue]@{-}[rd] \ar@[red]@{-}[uu]  \ar@[green]@{-}[ll]  &\\
1& & & & 2}}}

\newcommand{\Tetrahedron}{\vcenter{\xymatrix@=2.9ex{
*{\bullet} \ar@[red]@{-}[r] \ar@[green]@{-}[d]\ar@[blue]@{-}[rd] & *{\bullet}\ar@[blue]@{-}[ld]\\
*{\bullet} & *{\bullet}  \ar@[red]@{-}[l] \ar@[green]@{-}[u]}}}

\newcommand{\RedFrame}{\vcenter{\xymatrix@=2.9ex{
*{\bullet} \ar@[red]@{-}[r] \ar@[green]@<.15ex>@{-}[d]\ar@[blue]@<-.15ex>@{-}[d] & *{\bullet}\ar@[blue]@<.15ex>@{-}[d]\\
*{\bullet} & *{\bullet}  \ar@[red]@{-}[l] \ar@[green]@<.15ex>@{-}[u]}}}

\newcommand{\GreenFrame}{\vcenter{\xymatrix@=2.9ex{
*{\bullet} \ar@[green]@{-}[r] \ar@[red]@<.15ex>@{-}[d]\ar@[blue]@<-.15ex>@{-}[d] & *{\bullet}\ar@[blue]@<.15ex>@{-}[d]\\
*{\bullet} & *{\bullet}  \ar@[green]@{-}[l] \ar@[red]@<.15ex>@{-}[u]}}}

\newcommand{\BlueFrame}{\vcenter{\xymatrix@=2.9ex{
*{\bullet} \ar@[blue]@{-}[r] \ar@[red]@<.15ex>@{-}[d]\ar@[green]@<-.15ex>@{-}[d] & *{\bullet}\ar@[green]@<.15ex>@{-}[d]\\
*{\bullet} & *{\bullet}  \ar@[blue]@{-}[l] \ar@[red]@<.15ex>@{-}[u]}}}

\newcommand\blfootnote[1]{%
  \begingroup
  \renewcommand\thefootnote{}\footnote{#1}%
  \addtocounter{footnote}{-1}%
  \endgroup
}

\begin{document}

\title{Algebraic Methods for Tensor Data}

\author{Neriman Tokcan$^{1,2}$, Jonathan Gryak$^{3}$, Kayvan Najarian$^{3,4,5}$, and Harm Derksen$^{1,5}$}

\blfootnote{$^{1}$Department of Mathematics, University of Michigan, Ann Arbor}
\blfootnote{$^{2}$The Eli and Edythe L. Broad Institute of MIT and Harvard, Cambridge,  Massachusetts}
\blfootnote{$^{3}$Department of Computational Medicine and Bioinformatics, University of Michigan, Ann Arbor}
\blfootnote{$^{4}$Department of Emergency Medicine, University of Michigan, Ann Arbor}
\blfootnote{$^{5}$Michigan Center for Integrative Research in Critical Care, University of Michigan, Ann Arbor}

\maketitle

\begin{abstract}
We develop algebraic methods for computations with tensor data.
We give 3 applications: extracting features that are invariant under the orthogonal symmetries in each of the modes, approximation of the tensor spectral norm, and amplification of low rank tensor structure.
We introduce colored Brauer diagrams, which are used for algebraic computations and in analyzing their computational complexity. We present numerical experiments whose results show that the performance of the alternating least square algorithm for the low rank approximation of tensors can be improved using tensor amplification.

\end{abstract}

\section{Introduction}\label{sec:1}
Data in applications often is structured in higher dimensional arrays.
Arrays of dimension $d$ are also called {\em $d$-way tensors}, or {\em tensors of order $d$}.
It is challenging to generalize methods for matrices, which are 2-dimensional arrays, 
to tensors of order $3$ or higher. 
The notion of rank can be generalized from matrices to higher order tensors (see~\cite{Hitchcock1,Hitchcock2}).
Also, the spectral and nuclear norms are not only defined for matrices, but also for tensors of order $\geq 3$
(\cite{Grothendieck, Schatten}). However, the
rank, spectral norm, and nuclear norm of a higher order tensor are difficult to compute.
In fact, the related decision problems are NP-complete.
This was proved for the tensor rank in \cite{Hastad,Hastad2}, for the spectral norm in \cite{HL}
and for the nuclear norm in \cite{FL}.

We will use algebraic methods from classical invariant theory to perform various computations with tensors and analyze the computational complexity. Our methods
are based on the description of tensor invariants of the orthogonal group by {\em Brauer diagrams} (\cite{Brauer,GW,Weyl}). Brauer diagrams are perfect matching graphs. We will discuss the background on Classical Invariant Theory and Brauer diagrams in Section~\ref{sec:2}.
We will restrict ourselves to 3-way tensors. The techniques generalize to tensors of order $\geq 4$, but some of the formulas become more complicated. To perform computations with $3$-way tensors, we generalize the notion of Brauer diagrams to colored trivalent graphs called {\em colored Brauer diagrams}, Section~\ref{sec3}.

In this paper we consider 3 applications of our algebraic approach, namely invariant tensor features from data,
approximations of the spectral and nuclear norm, and tensor amplification. In Subsection~\ref{sec4.1}, we introduce  the norm $|| \mathcal{T}||_{\sigma,m}^{m}$  for $m \in \mathbb{N}$ to approximate the spectral norm of 3-way tensors. In Subsection \ref{sec4.2}, we show that $|| \mathcal{T}||_{\sigma,2}$ is equal to the Euclidean norm (alternatively called Frobenius, Hilbert-Schmidth norm). In Subsection \ref{sec4.4}, we introduce another norm $||T||_{\#}$ that approximates the spectral norm.  The main results are explicit formulas of these norms in terms of colored Brauer diagrams (see Theorem~\ref{theo:approxd4} and Proposition~\ref{prop:IsNorm}) and a comparison between the spectral norm and these approximations (see~Proposition~\ref{prop:NormBetter}) is given.
In Section~\ref{sec5}, we study the low rank amplification methods based on the approximations of the spectral norm. We employ these amplification methods to obtain better initial guesses for the CP-ALS method; an algorithm for the low rank approximation to $3$-way tensors is given (Section~\ref{ALS}, Algorithm~5.1). In Section~\ref{experiment}, we compare the ALS tensor approximation based on tensor amplification initialization with random initialization. In our experiments, we see that methods introduced in Section \ref{ALS} give low rank $r$ approximations (r=1 in Subsection \ref{rank1test} and  r=2 in Subsection \ref{rank2test}) with better fits and improved time efficiency compared to CP-ALS method. 

\subsection{Notation and Preliminaries}\label{notation}

We will introduce the basic concepts and notation, which will lay the foundation for the rest of the paper. We will borrow most of our notation from \cite{FL} and \cite{KB}.

As we have stated before, tensors are multi-dimensional arrays. The \textit{order}  of a tensor is the number of its dimensions (ways, modes). Vectors are tensors of order 1 and matrices are tensors of order 2. We will refer to tensors of order 3 or higher as \text{higher-order tensors}.  Vectors are denoted by lower case letters $x  \in \mathbb{R}^{p}$, matrices are denoted by capital letters  $X \in \mathbb{R}^{p  \times q},$ and higher-order tensors are denoted by capital calligraphic letters $\mathcal{X} \in \R^{p_1 \times p_2 \times \ldots \times p_d}.$  The $(i_1,i_2,\ldots, i_d)-$th entry of the $d$-th order ($d$-way) tensor $\mathcal{X}$ is denoted by $x_{i_{1}i_{2}\ldots i_{d}}.$ 

The \textit{vector outer product} of $u \in \mathbb{R}^p~\text{and}~v \in \mathbb{R}^q$ is denoted by $u \otimes v$  and it can be given as the matrix multiplication $uv^{T} \in \mathbb{R}^{p \times q}.$

The \textit{inner product} of two same size tensors $\mathcal{X}, \mathcal{Y} \in \mathbb{R}^{p_{1}\times p_{2} \times \ldots \times  p_{d}}$ is defined as follows: 
\begin{equation}
\mathcal{X} \cdot \mathcal{Y} = \sum_{i_1=1}^{p_1} \sum_{i_2=1}^{p_2}\ldots \sum_{i_d=1}^{p_d} x_{i_1i_2\ldots i_d}y_{i_1i_2\ldots i_d}  \in \mathbb{R}.
\end{equation}
It follows immediately that the norm of a tensor is the square root of the sum of the squares of all its elements: 
\begin{equation}
     \|\mathcal {X}\|= \sqrt{\sum_ {i_1=1}^{p_1}\sum_{i_2=1}^{p_2}\ldots \sum_{i_d=1}^{p_d}x_{i_1i_2\ldots i_d}^2}.
\end{equation}

This is analogous to the matrix Frobenius norm, see Section \ref{spectral-nuclear} for more details on the Frobenius norm.  

A tensor $ \mathcal{S} \in \mathbb{R}^{p_1\times p_2 \times \ldots \times p_d}$ is \textit{rank one} if it can be written as an outer product of $n$ vectors, i.e., $\mathcal{S}=u_{1}\otimes u_{2}\otimes \ldots \otimes u_{d},~u_{i} \neq 0 \in \mathbb{R}^{p_i},~1\leq i\leq d.$ Such rank one tensors are also called \textit{simple} or \textit{pure}.

The best rank 1 approximation problem of a tensor $\mathcal{T} \in \mathbb{R}^{p_1\times p_2\times\ldots\times p_d}$ can be stated as follows:
\begin{equation}\label{rank1}
    \min_{\mathcal{S}} \| \mathcal {T} - \mathcal{S} \|~\text{where}~\mathcal{S}~\text{is a rank one tensor in}~\mathbb{R}^{p_1\times p_2\times\ldots\times p_d}.
\end{equation} 

The best rank 1 approximation problem is well-posed and NP-hard (\cite{Landsberg, SCA}). Different algebraic tools and algorithms  have been proposed to find the global minimum of Problem \ref{rank1} (see \cite{SCA, ZG}).

A tensor $\mathcal{S} \in \mathbb{R}^{p_1 \times p_2 \times \ldots \times p_d}$ can be represented as a linear combination of rank 1 tensors:

\begin{equation}\label{decomp}
    \mathcal{S}=\sum_{i=1}^{r} \lambda_i u_{1,i}\otimes u_{2,i}\otimes \ldots \otimes u_{d,i}
\end{equation}
where $r$ is sufficiently large, $\lambda_i \in \mathbb{R}$ and $u_{j,i} \in \mathbb{R}^{p_j}$ for $1 \leq i \leq r,~1\leq j\leq d.$ The smallest such integer $r$ is called the \textit{rank~(real rank)} of the tensor. The decomposition given in \eqref{decomp} is often referred to as the \textit{rank r decomposition, CP (Candecomp / Parafac), or Canonical Polyadic decomposition}. Let $U^{(j)}=[u_{j,1}u_{j,2}\ldots u_{j,r}] \in \mathbb{R}^{p_j \times r},~1\leq j \leq d.$ We call these matrices as \textit{factor matrices}. Then CP decomposition factorizes a $d$-way tensor into $d$ factor matrices and a vector $\Lambda=[\lambda_1, \lambda_2,\ldots, \lambda_r] \in \mathbb{R}^{r}.$ The decomposition in \eqref{decomp} can be concisely expressed as ${\mathcal S}=\llbracket \Lambda~;~U^{(1)},~U^{(2)},\ldots,~U^{(d)}\rrbracket.$ As in  \eqref{rank1}, the best rank $r$ approximation problem for a  tensor $\mathcal{T}  \in \mathbb{R}^{p_1 \times p_2 \times \ldots \times p_d}$ can be given as:\begin{equation}\label{rankr}
\min_{ \Lambda, U^{(1)}, \ldots, U^{(d)}} \|{\mathcal T - \mathcal{S}}\|~\text{where}~\mathcal{S}=\llbracket \Lambda~;~U^{(1)},~U^{(2)},\ldots,~U^{(d)} \rrbracket.
\end{equation}

The solution to Problem (\ref{rankr}) does not always exist (\cite{KB, SL}). Alternating Least Squares (ALS) is the most common method used for the low rank approximation, since it is simple and easy to implement. However, it has some limitations: convergence is slow for some cases, it is heavily dependent on the initial guess of the factor matrices, and it may not converge to a global minimum (see \cite{KB, Landsberg} for more details on the CP decomposition and ALS method). More details on the ALS algorithm for the low rank approximation are given in Section \ref{ALS}.

\subsection{Invariant tensor features from data}
Let $\OO_p(\R)$ be the group of orthogonal $p\times p$ matrices. The group $\OO_p(\R)\times \OO_q(\R)$ acts
on the space $\R^{p\times q}$ of $p\times q$ matrices by left and right multiplication. A group element $(B,C)\in \OO_p(\R)\times \OO_q(\R)$ acts on a matrix $A\in \R^{p\times q}$ by $(B,C)\cdot A=BAC^{-1}$.
The singular values $\lambda_1(A)\geq \lambda_2(A)
\geq \dots \geq  \lambda_r(A) \geq 0$ of a $p\times q$ matrix $A$ are the features that are invariant under
the actions of $\OO_p(\R)$ and $\OO_q(\R)$ on the rows and columns respectively. In other words, if $B$ and $C$ are orthogonal matrices, then  $\lambda_i(BAC^{-1})=\lambda_i(A)$ for $1 \leq i \leq r.$
The function $t_k$ given by  $t_k(A)=\tr((AA^{T})^k)=\lambda_1(A)^{2k}+\lambda_2(A)^{2k}+\cdots+\lambda_r(A)^{2k},~k\geq 0$
is also invariant under the actions of $\OO_p(\R)$ and $\OO_q(\R)$. The invariant functions $t_1(A),t_2(A),\dots$ are polynomials in the entries of the matrix $A$. It is known that the set $\{t_{k} : k \geq 0 \}$ generates the ring of polynomial invariants under the action of  $\OO_p(\R)\times \OO_q(\R)$ on $p\times q$ matrices (see for example \cite[\S12.4.3, type {\bf BDI}]{GW}, but here we do not need any Pfaffians because we consider orthogonal groups and not special orthogonal groups).
We will consider invariant features for $3$-way tensors. Using classical invariant theory for the orthogonal group, we will describe polynomial tensor invariants in terms of colored Brauer diagrams. A similar approach to describing tensor invariants of orthogonal group actions can be found in the thesis \cite[\S4.2]{Williams}. We will introduce the colored Brauer diagrams in Section~\ref{sec3} and use them to construct polynomial tensor invariants.

\subsection{Approximations of the spectral and nuclear norm}\label{spectral-nuclear}
Important norms on the space of $p\times q$ matrices are the Frobenius norm (or the Euclidean $\ell_2$-norm), the spectral norm (or operator norm), and the nuclear norm.
These norms can be expressed in terms of the singular values of a matrix.
If $\lambda_1\geq \lambda_2\geq \dots\geq\lambda_r\geq 0$ are the singular values of a matrix $A$,
then the Frobenius norm is $\|A\|=\|A\|_F=\sqrt{\lambda_1^2+\lambda_2^2+\cdots+\lambda_r^2}$, the
spectral norm is $\|A\|_\sigma=\lambda_1$, and the nuclear norm is $\|A\|_\star=\lambda_1+\lambda_2+\cdots+\lambda_r$. The nuclear norm can be seen as a convex relaxation of the rank of a matrix. It is used for example in some algorithms for the  matrix completion problem which asks to complete a partially filled matrix such that the rank of the resulting matrix has minimal rank (\cite{EB,Convexrelax}). This problem has applications to collaborative filtering (see~\cite{Collaborative}). The spectral and nuclear norms generalize to higher order tensors. 

Let ${\mathcal T}$ be a tensor of order $d$ in $\mathbb{R}^{p_1 \times p_2 \times \ldots\times p_d}.$ We define its \textit{spectral norm} by
\begin{equation}
  \textstyle  \| {\mathcal T}\|_{\sigma}=  \sup \left \{ | {\mathcal T}\cdot u_{1} \otimes u_{2} \otimes \ldots \otimes u_{d} |~:~u_{j} \in \mathbb{R}^{p_j}, \|u_{j}\|=1,~1\leq j\leq d \right \}.
\end{equation}
It is known that the dual of the spectral norm is the \textit{nuclear norm} and it can be defined as 

\begin{multline}
  \textstyle \| \mathcal{T} \|_{\star} =\inf \left\{ \sum_{i=1}^r|\lambda_i|~:~\mathcal{T}=\sum_{i=1}^{r} \lambda_i u_{1,i}\otimes u_{2,i}\otimes\ldots\otimes u_{d,i},~\text{where}~\lambda_i \in \mathbb{R},\right.\\
 \textstyle u_{j,i} \in \left. \mathbb{R}^{p_j},~
 \|u_{j,i}\|=1,~1\leq j\leq d,~1\leq i\leq r 
\right\}.
\end{multline}

These generalizations are more difficult to compute,
as the corresponding decision problems are NP-hard (\cite{FL,HL}). As in the matrix case, nuclear norm of a tensor is considered as a convex relaxation of the tensor rank \cite{Derksen}. The nuclear and spectral norms of  tensors play an important role in tensor completion problems \cite{MZ}. Different methods to  estimate and to evaluate the  spectral norm and the nuclear norm and their upper and lower bounds have been studied by several authors (see \cite{Xu,Li, LNS, LS}). 

The spectral norm is related to rank 1 approximation of a given tensor. If ${\mathcal S}$ is a best rank 1 approximation of a given tensor $\mathcal{T},$ then $\|\mathcal{T}- \mathcal {S}\|=\sqrt{\|\mathcal{T}\|^2-\|\mathcal{T}\|_\sigma^2}$, (Proposition 1.1, \cite{LNS}).

We will give approximations of the spectral norm that can be computed in polynomial time using colored Brauer diagrams in Section~\ref{sec4}. For every even $d$ we define a norm $\|\cdot\|_{\sigma,d}$ that approximates the spectral norm $\|\cdot\|_{\sigma}$
such that $\|\cdot\|_{\sigma,d}^d$ is a polynomial function of degree $d$ and $\lim_{d\to\infty} \|{\mathcal T}\|_{\sigma, d}=\|{\mathcal T}\|_\sigma$ for any tensor ${\mathcal T}$. 
One of our main results is an explicit formula for the norm $\|\cdot\|_{\sigma,4}$ for tensors of order 3 in terms of colored Brauer diagrams (see Theorem~\ref{theo:approxd4}), which allows us to compute this norm efficiently. We also introduce another norm 
$\|\cdot\|_\#$ (see Definition~\ref{def:NewNorm}, Proposition~\ref{prop:IsNorm}) and show that it is, in some sense, a better approximation to the spectral norm than $\|\cdot\|_{\sigma,4}$ (see Proposition~\ref{prop:NormBetter}).

$\|. \|$ will stand for the Frobenius norm throughout the paper. 

\subsection{Tensor amplification}\label{tensor-amplification}
If $A$ is a real matrix with singular values $\lambda_1,\dots,\lambda_r$, then the matrix $AA^{T} A$
has singular values $\lambda_1^3,\lambda_2^3,\dots,\lambda_r^3$.
The map $A\mapsto AA^{T} A$ has the effect of amplifying the low rank structure corresponding to larger singular values, while suppressing the smaller singular values that
typically correspond to noise. Using colored Brauer diagrams, we will construct similar amplification maps for tensors of order 3 in Section~\ref{sec5}. We also will present numerical experiments whose results show that tensor amplification can reduce the running time of the alternating least square algorithm for the low rank tensor approximation, while producing a better approximation.


\section{Brauer diagrams}\label{sec:2}
In this section, we will give an overview of the classical invariant theory of the orthogonal group. We recall the relation between Brauer diagrams and invariant tensors for the orthogonal group.
\subsection{Orthogonal transformations on tensors}

Let $V\cong \R^n$ be a Euclidean vector space with basis $\e_1,\e_2,\dots,\e_n$. The orthogonal group  $\OO(V)=\OO_n(\R)=\{A\in \R^{n\times n}\mid AA^{T}=I\}$ acts on $V$. On $V$ we have an inner product that allows us to identify $V$ with its dual space $V^\star$. We consider the $d$-fold tensor product of $V$:
\begin{equation}
V^{\otimes d}=\underbrace{V\otimes V\otimes\cdots\otimes V}_d\cong\R^{n\times n\times \cdots\times n}\cong \R^{n^d}.
\end{equation}
There are various ways to think of elements of $V^{\otimes d}$. The following statement is well known (Chapter 2, \cite{Landsberg}).

\begin{lemma}\label{lem:TensorBijections}
There are bijections between the following sets:
\begin{enumerate}
    \item the set of {\bf tensors} $V^{\otimes d}$;
    \item $(V^{\otimes d})^\star$, the set of {\bf linear maps} $V^{\otimes d}\to \R$;
    \item the set of {\bf $\R$-multilinear maps} $V^d\to \R$.
    \end{enumerate}
\end{lemma}
\begin{proof}
We have a multi-linear map $\iota:V^d\to V^{\otimes d}$ given by $\iota(v_{1},v_{2},\dots,v_{d})=v_{1}\otimes v_{2}\otimes \cdots \otimes v_{d},$ where $v_{i} \in V$ for $i=1,\ldots,d.$ 
Any linear map $L:V^{\otimes d}\to \R$ induces a multi-linear map $\ell=L\circ \iota:V^d\to \R$. Conversely, every multi-linear map $\ell:V^d\to \R$ factors 
as $\ell=L\circ \iota$ for a unique linear map $L:V^{\otimes d}\to \R$ by the universal property of the tensor product (see~\cite[Chapter XVI]{Lang}). This proves the bijection between (2) and (3).
 Since we have identified $V$ with its dual $V^\star$ we can also identify $V^{\otimes d}$ with $(V^\star)^{\otimes d}\cong (V^{\otimes d})^\star$, which gives the equivalence between (1) and (2).
\end{proof}
We will frequently switch between these different viewpoints in the lemma.

The group $\OO(V)$ and the symmetric group $\Sigma_d$ act on the $d$-fold tensor product space
as follows. Let $S$ be a rank $r$ tensor in $V^{\otimes d}$ such that ${\mathcal S}=\sum_{i=1}^rv_{1,i}\otimes v_{2,i}\otimes \cdots\otimes v_{d,i}\in V^{\otimes d},$ where  $v_{j,i} \in V$ for all $i=1,\ldots,r$ and $j=1,\ldots,d$. 
 If $A\in \OO(V),$ then we have
\begin{equation} 
\textstyle A\cdot {\mathcal S}=\sum_{i=1}^r Av_{1,i}\otimes Av_{2,i}\otimes \cdots\otimes Av_{d,i}.
\end{equation} 
If $\pi\in \Sigma_d$ is a permutation, then
\begin{equation} \label{eq:symmaction}
\textstyle \pi\cdot {\mathcal S}=\sum_{i=1}^r v_{\pi^{-1}(1),i}\otimes v_{\pi^{-1}(2),i}\otimes \cdots\otimes v_{\pi^{-1}(d),i}.
\end{equation} 
The actions of $\OO_n(\R)$ and $\Sigma_d$ on $V^{\otimes d}$ commute.

The subspace of $\OO(V)$-invariant tensors in $V^{\otimes d}$ is
\begin{equation}
(V^{\otimes d})^{\OO(V)}=\{{\mathcal T}\in V^{\otimes d} : A\cdot {\mathcal T}={\mathcal T}\mbox{ for all $A\in \OO(V)$}\}.
\end{equation}
A linear map $L:V^{\otimes d}\to \R$ is $\OO(V)$-invariant
if $L(A\cdot {\mathcal T})=L({\mathcal T})$ for all tensors ${\mathcal T}$ and all $A\in \OO(V)$.
A multi-linear  map $M:V^d\to \R$ is $\OO(V)$-invariant if
$M(Av_{1},\dots,Av_{d})=M(v_{1},\dots,v_{d})$ for all $v_{1},\dots,v_{d}\in V$ and all $A\in \OO(V)$.

\begin{corollary}\label{cor:InvariantTensors}
There are bijections between the following sets:
\begin{enumerate}
    \item $(V^{\otimes d})^{\OO(V)}$, the set of $\OO(V)$-invariant tensors in $V^{\otimes d}$;
    \item the set of $\OO(V)$-invariant linear maps $V^{\otimes d}\to \R$;
    \item the set of $\OO(V)$-invariant multilinear maps $V^d\to \R$.
\end{enumerate}
\end{corollary}
\begin{proof}
Following the proof of Lemma~\ref{lem:TensorBijections}, we see that the bijections in Lemma~\ref{lem:TensorBijections} preserve the action of the orthogonal group $\OO(V)$ and induce the desired bijections in Corollary~\ref{cor:InvariantTensors}.
\end{proof}

\subsection{The First Fundamental Theorem of Invariant Theory}

The First Fundamental Theorem of Invariant Theory for the orthogonal group (Theorem~\ref{thm:FFT} below) gives us a description of $(V^{\otimes d})^{\OO(V)}$.
If $d$ is odd then $(V^{\otimes d})^{\OO(V)}=0$.
We now will describe $(V^{\otimes d})^{\OO(V)}$ for $d=2e$ where $e$ is a positive integer.

A {\bf labeled Brauer diagram} of size $d=2e$ is a perfect matching of a complete graph where the vertices are labeled $1,2,\dots,d$ (see~\cite[Chapter 10]{GW} for more details).
\begin{example}
Below is a labeled Brauer diagram of size $6$:
\begin{equation}\label{eq:0}
{\bf D}=\vcenter{\xymatrix@=2.9ex{
1 \ar@{-}[r] & 3 & 5\ar@{-}[ld]\\
2\ar@{-}@/_/[rr] & 4 & 6
}}.
\end{equation}
We denote this diagram by $(1\ 3)(2\ 6)(4\ 5)$.
\end{example}

To a labeled Brauer diagram {\bf D} of size $d=2e$ we can associate an $\OO(V)$-invariant multi-linear map
${\mathcal M}_{\bf D}:V^{d}\to \R$
as follows. If $i_k$ is connected to $j_k$ for $k=1,2,\dots,e$ in the diagram {\bf D}, then we define
\begin{equation}
{\mathcal M}_{\bf D}(v_{1},v_{2},\dots,v_{d})= (v_{i_1}\cdot v_{j_1})(v_{i_2}\cdot v_{j_2})\cdots (v_{i_e}\cdot v_{j_e})
\end{equation}
for all $v_{1},\dots,v_{d}\in V$. By Corollary~\ref{cor:InvariantTensors}
the $\OO(V)$-invariant multilinear map ${\mathcal M}_{\bf D}$ 
corresponds to some $\OO(V)$-invariant linear map ${\mathcal L}_{\bf D}:V^{\otimes d}\to \R$ and an $\OO(V)$-invariant tensor ${\mathcal T}_{\bf D}\in (V^{\otimes d})^{\OO(V)}$, which we make more explicit now. As in the proof of Lemma~\ref{lem:TensorBijections}, the universal property of the tensor product gives us a unique linear map ${\mathcal L}_{\bf D}:V^{\otimes d}\to \R$ such that 
\begin{equation}
{\mathcal L}_{\bf D}(v_{1}\otimes v_{2}\otimes \cdots\otimes v{d})={\mathcal M}_{\bf D}(v_{1},v_{2},\dots,v_{d})=(v_{i_1}\cdot v_{j_1})(v_{i2}\cdot v_{j_2})\cdots (v_{i_e}\cdot v_{j_e}).
\end{equation}
By Corollary~\ref{cor:InvariantTensors}, there is also a unique tensor ${\mathcal T}_{\bf D}\in V^{\otimes d}$ such that 
${\mathcal L}_{\bf D}(A)={\mathcal T}_{\bf D}\cdot A$ for all tensors $A\in \OO(V)$.

\begin{example}
If {\bf D} is the diagram in (\ref{eq:0}), and $\e_1,\dots,\e_n$ is a basis of $V$, then we have
\begin{equation}
{\mathcal T}_{\bf D}=\sum_{i=1}^n\sum_{j=1}^n\sum_{k=1}^n \e_i\otimes \e_j\otimes \e_i\otimes \e_k\otimes \e_k\otimes \e_j.
\end{equation}
The indices $i,j,k$ correspond to the edges $(1\ 3)$, $(2\ 6)$ and $(4\ 5)$ respectively.
\end{example}
The proof of the following theorem is in Theorem 4.3.3 and  Proposition 10.1.3 of \cite{GW}. 
\begin{theorem}[FFT of Invariant Theory for $\OO_n$ \cite{GW,PV}]\label{thm:FFT}
The space $(V^{\otimes d})^{\OO(V)}$ of invariant tensors is spanned by all ${\mathcal T}_{\bf D}$ where ${\bf D}$ is a Brauer diagram on $d$ vertices.
\end{theorem}
The following result is well-known (see for example~\cite{Callan, OEISa}), but the idea of the proof is useful later.

\begin{proposition}\label{prop1}
The number of
Brauer diagrams (and perfect matchings in a complete graph) for $d=2e$ vertices is $1\cdot 3\cdot 5\cdots (2e-1)$. 
\end{proposition}
\begin{proof}
 Let $N_e$ be the number of Brauer diagrams on $2e$ nodes.
 Clearly $N_1=1$. We can construct $2e+1$ Brauer diagrams on $2e+2$ nodes from a Brauer diagram ${\bf D}$ on $2e$ nodes as follows.
We take ${\bf D}$ and add two nodes, $2e+1$ and $2e+2$. 
First, we can choose an integer $k$ with $1\leq k\leq 2e$ and let $l$ be the vertex that $k$ is connected to. Then we disconnect $k$ from $l$,
connect $2e+1$ to $k$ and connect $2e+2$ to $l$. This gives us a Brauer diagram ${\bf D}_k$ on $2e+2$ nodes.
Alternatively, we can also  connect $2e+1$ to $2e+2$ and get a Brauer diagram on $2e+2$ nodes that we call ${\bf D}_{2e+1}$.
Thus, we have constructed Brauer diagrams ${\bf D}_1,\dots,{\bf D}_{2e+1}$ from ${\bf D}$.
One can verify that we generate all Brauer diagrams on $2e+2$ nodes exactly once if we vary ${\bf D}$ over all Brauer diagrams on $2e$ nodes.
So $N_{e+1}=(2e+1)N_e$ for all $e$.
\end{proof}

\subsection{Partial Brauer diagrams}
A {\bf partial Brauer diagram} of size $d$ is a graph with $d$ vertices labeled $1,2,\dots,d$ whose edges form a partial matching.
Our convention is to draw loose edges at the vertices that are not matched.
To a partial Brauer diagram ${\bf D}$ with $e$ edges, we can associate an $\OO(V)$-invariant  multi-linear map ${\mathcal M}_{\bf D}:V^d\to V^{\otimes (d-2e)}$
and a linear map ${\mathcal L}_{\bf D}:V^{\otimes d} \to V^{\otimes d-2e}$. 
\begin{example}
For the diagram:
\begin{equation}
{\bf D}=\vcenter{\xymatrix@=2.9ex{
1 \ar@{-}[r] & 3 & 5\ar@{-}[ld]\\
2\ar@{-}[d] & 4 & 6\ar@{-}[d]\\
& &
}}
\end{equation}
we have
\begin{equation}
{\mathcal M}_{\bf D}(v_{1},v_{2},\dots,v_{6})= (v_{1}\cdot v_{3})(v_{4}\cdot v_{5}) v_{2}\otimes v_{6}\in V^{\otimes 2}
\end{equation}
for $v_{1},v_{2},\dots,v_{6}\in V$. 
\end{example}
Before giving the general rule of computing inner products of tensors associated to Brauer diagrams, we give an illustrative example.
\begin{example}
We compute the inner product ${\mathcal T}_{{\bf D}_1}$ and ${\mathcal T}_{{\bf D}_2}$ where ${\bf D}_1$ and ${\bf D}_2$ are the diagrams below:
\begin{equation}
{\bf D}_1=\vcenter{\xymatrix@=2.9ex{
1 \ar@{-}[r] & 3 & 5\ar@{-}[ld]\\
2\ar@{-}@/_/[rr] & 4 & 6
}}\qquad
{\bf D}_2=\vcenter{\xymatrix@=2.9ex{
1 \ar@{-}[d] & 3\ar@{-}[rd] & 5\ar@{-}[ld]\\
2 & 4 & 6
}}.
\end{equation}
We get
\begin{multline}
{\mathcal T}_{{\bf D}_1}\cdot {\mathcal T}_{{\bf D}_2}=\left(\sum_{i,j,k} \e_i\otimes \e_j\otimes \e_i \otimes \e_k\otimes \e_k\otimes \e_j\right)\cdot
\left(\sum_{p,q,r} \e_p\otimes \e_p\otimes \e_q\otimes \e_r\otimes \e_r\otimes \e_q\right)=\\
\sum_{i,j,k,p,q,r}(\e_i\cdot \e_p)(\e_j\cdot \e_p)(\e_i\cdot \e_q)(\e_k\cdot \e_r)(\e_k\cdot \e_r)(\e_j\cdot \e_q).
\end{multline}
To get a nonzero summand we have to have $i=j=p=q$ and $k=r$.
The result of the summation is $\sum_{i=1}^n\sum_{k=1}^n 1=n^2$. We can visualize this computation as follows
\begin{equation}
\vcenter{\xymatrix@=2.9ex{
*{\bullet} \ar@{-}[r] & *{\bullet} & *{\bullet}\ar@{-}[ld]\\
*{\bullet} \ar@{-}@/_/[rr] & *{\bullet} & *{\bullet}
}}\quad\cdot \quad
\vcenter{\xymatrix@=2.9ex{
*{\bullet}\ar@{-}[d]& *{\bullet}\ar@{-}[rd] & *{\bullet}\ar@{-}[ld]\\
*{\bullet} & *{\bullet} & *{\bullet}
}}\quad =\quad
\vcenter{\xymatrix@=2.9ex{
*{\bullet} \ar@{-}[r] \ar@{-}[d]  & *{\bullet}\ar@{-}[rd] & *{\bullet}\ar@{-}@<.2ex>[ld]\ar@{-}@<-.2ex>[ld]\\
*{\bullet} \ar@{-}@/_/[rr] & *{\bullet} & *{\bullet}
}} \quad=n^2.
\end{equation}
The edges of ${\bf D}_1$ correspond to the indices $i,j,k$ and the edges of ${\bf D}_2$ correspond to the indices $p,q,r$. We overlay the diagrams. 
The indices of the edges in a cycle must all be the same. Since there are two cycles, namely $i,p,j,q$ and $k,r$ we essentially sum over two indices and get $n^2$.
\end{example}
The general rule is clear now.
\begin{corollary}
The dot product of two tensors ${\mathcal T}_{{\bf D}_1},{\mathcal T}_{{\bf D}_2}\in V^{\otimes d}$  can be computed as follows.
We overlay the two diagrams ${\bf D}_1$ and ${\bf D}_2$ so that the (labeled) nodes coincide. Then ${\mathcal T}_{{\bf D}_1}\cdot {\mathcal T}_{{\bf D}_2}=n^k$ where $k$ is the number of cycles (including $2$-cycles).
\end{corollary}
\begin{proof}
The tensor ${\mathcal T}_{{\bf D}_1}$ is equal
to the sum of all
$\e_{i_1}\otimes \e_{i_2}\otimes \cdots\otimes \e_{i_{d}}$
with $1\leq i_1,i_2,\dots,i_d\leq n$
such that $i_j=i_k$ whenever $(j\ k)$ is an edge in ${\bf D}_1$.
Similarly, ${\mathcal T}_{{\bf D}_2}$ is the sum of all
$\e_{p_1}\otimes \e_{p_2}\otimes \cdots\otimes \e_{p_{d}}$
with $1\leq p_1,p_2,\dots,p_d\leq n$
and $p_j=p_k$ whenever $(j\ k)$ is an edge in ${\bf D}_2$. We compute
\begin{equation}\label{eq:overlay}
{\mathcal T}_{{\bf D}_1}\cdot {\mathcal T}_{{\bf D}_2}=\sum_{i_1,\dots,i_d,p_1,\dots p_d} (\e_{i_1}\cdot \e_{p_1})(\e_{i_2}\cdot \e_{p_2})\cdots (\e_{i_d}\cdot \e_{p_d}),
\end{equation}
where the sum is over all $(i_1,\dots,i_d,p_1,\dots,p_d)$ for which
$i_j=i_k$ when $(j\ k)$ is an edge in ${\bf D}_1$ and $p_j=p_k$ when $(j\ k)$ is an edge in ${\bf D}_2$.
The summand $(\e_{i_1}\cdot \e_{p_1})(\e_{i_2}\cdot \e_{p_2})\cdots (\e_{i_d}\cdot \e_{p_d})$ is equal to 1
if $i_j=p_j$ for all $j$,
and $0$ otherwise. Setting $i_j$ equal to $p_j$ corresponds to overlaying the diagrams ${\bf D}_1$ and ${\bf D}_2$. 
So (\ref{eq:overlay}) is equal to the number of
tuples $(i_1,i_2,\dots,i_d)$
with $1\leq i_1,i_2,\dots,i_d\leq n$,
and $i_j=i_k$ whenever $(j\ k)$ is an edge in ${\bf D}_1$ {\em or} in ${\bf D}_2$.
If $k$ is the number of cycles,
then we can choose exactly $k$ of indices
$i_1,i_2,\dots,i_d$ freely in the set $\{1,2,\dots,n\}$, and the other indices are uniquely determined by these choices.
This proves that (\ref{eq:overlay}) is equal to $n^k$.
\end{proof}

The proof of the following proposition is similar to the proof of Proposition~\ref{prop1}.
\begin{proposition}
Suppose that ${\bf E}$ is a Brauer diagram on $d=2e$ vertices, and  ${\mathcal S}_d=\sum_{\bf D} {\mathcal T}_{\bf D}$,
where the sum is over all Brauer diagrams ${\bf D}$ on $d$ vertices. 
Then we have ${\mathcal T}_{\bf E}\cdot {\mathcal S}_d=n(n+2)\cdots(n+d-2)$.
\end{proposition}
\begin{proof}
Let  {\bf E} be a Brauer diagram on $d=2e$ vertices and let ${\bf E}'$ be the diagram obtained from ${\bf E}$ by adding the edge $(d+1\ d+2)$.
 Recall that in the proof of Proposition~\ref{prop1}, we constructed diagrams ${\bf D}_1,{\bf D}_2,\dots,{\bf D}_{d+1}$ from
 the Brauer diagram {\bf D} on $d$ vertices.
 Note that if $1\leq i\leq d$, we get ${\mathcal T}_{{\bf E}'}\cdot {\mathcal T}_{{\bf D}_i}={\mathcal T}_{\bf E}\cdot {\mathcal T}_{{\bf D}}$ because we get the diagram of ${\mathcal T}_{{\bf E}'}\cdot {\mathcal T}_{{\bf D}_i}$ from  that of ${\mathcal T}_{{\bf E}}\cdot {\mathcal T}_{\bf D}$
 by changing one $k$-cycle to a $(k+2)$-cycle. We also have ${\mathcal T}_{{\bf E}'}\cdot {\mathcal T}_{{\bf D}_{d+1}}=n({\mathcal T}_{\bf E}\cdot {\mathcal T}_{\bf D})$ as we are adding one $2$-cycle.
 This shows that ${\mathcal T}_{{\bf E}'}\cdot {\mathcal S}_{d+1}=(n+d){\mathcal T}_{\bf E}\cdot {\mathcal S}_d.$ The proposition then follows by induction and symmetry.
 \end{proof}
 \begin{example}
 For $e=2$, we get
 \begin{equation}
\IIblack\quad\cdot\quad\left( \IIblack+\Zblack+\Xblack\right)=
\vcenter{\xymatrix@=2.9ex{
*{\bullet}\ar@{-}@<.2ex>[d] \ar@{-}@<-.2ex>[d] & *{\bullet} \ar@{-}@<.2ex>[d] \ar@{-}@<-.2ex>[d]\\
*{\bullet} & *{\bullet}}}+
\vcenter{\xymatrix@=2.9ex{
*{\bullet}\ar@{-}[d]\ar@{-}[r] & *{\bullet}\ar@{-}[d]\\
*{\bullet}\ar@{-}[r] & *{\bullet}}}+\vcenter{\xymatrix@=2.9ex{
*{\bullet}\ar@{-}[rd] \ar@{-}[d]& *{\bullet}\ar@{-}[ld]\ar@{-}[d]\\
*{\bullet} & *{\bullet}}}
=n^2+n+n=n(n+2).
 \end{equation}
 \end{example}
 \subsection{The expected rank 1 unit tensor}
Let $S^{n-1}=\{v\in V\mid \|v\|=1\}$ be the unit sphere equipped with the $\OO(V)$-invariant volume form $d\mu$ that is normalized
such that $\int_{S^{n-1}}d\mu=1$.

\begin{proposition}\label{integral}
If we integrate 
\begin{equation}
v^{\otimes 2e}=\underbrace{v\otimes v\otimes \cdots \otimes v}_{2e}
\end{equation} over $S^{n-1}$ then we get
$\int_{S^{n-1}} v^{\otimes 2e}\,d\mu={\textstyle  \frac{1}{n(n+2)\cdots (n+2e-2)}}{\mathcal S}_{2e}$.
\end{proposition}
\begin{proof}
Let $U=\int_{S^{n-1}} v^{\otimes 2e}\,d\mu$. Since $U$ is $\OO(V)$-invariant, it is a linear combination of Brauer diagrams. The tensor $U$ 
is also invariant under the action of the symmetric group 
$\Sigma_{2e}$, where the action of the symmetric group is given in (\ref{eq:symmaction}). This shows that each Brauer diagram appears with the same coefficient in $U$.
So we have $U=C {\mathcal S}_{2e}$ where $C$ is some constant. Let ${\bf D}$ be some Brauer diagram on $2e$ vertices. 
 The value of $C$ is obtained from
\begin{multline}
1=\int_{S^{n-1}}d\mu=\int_{S^{n-1}} ({\mathcal T}_{\bf D}\cdot v^{\otimes 2e})d\mu={\mathcal T}_{\bf D}\cdot \int_{S^{n-1}} v^{\otimes 2e}\,d\mu=\\=C ({\mathcal T}_{\bf D}\cdot {\mathcal S}_{2e})=Cn(n+2)\cdots (n+2e-2).
\end{multline}
\end{proof}

\section{Computations with 3-way tensors}\label{sec3}
\subsection{Colored Brauer diagrams}
We now consider 3 Euclidean $\R$-vector spaces $R$, $G$ and $B$ of dimension $p$, $q$ and $r$ respectively. 
The tensor product space $V=R\otimes G\otimes B$ is a representation of $H:=\OO(R)\times \OO(G)\times \OO(B)$. We keep the notations introduced in Sections \ref{sec:1} and \ref{sec:2} based on the fact that tensor product of vector spaces is a vector space itself.
We are interested in $H$-invariant tensors in $V^{\otimes d}$. 
We have an explicit linear isomorphism $\psi:R^{\otimes d}\otimes G^{\otimes d}\otimes B^{\otimes d}\to V^{\otimes d}$ defined by
\begin{equation}
\psi\big((a_1\otimes \cdots\otimes a_d)\otimes (b_1\otimes \cdots\otimes b_d)\otimes (c_1\otimes \cdots \otimes c_d)\big)= (a_1\otimes b_1\otimes c_1)\otimes \cdots \otimes (a_d\otimes b_d\otimes c_d),
\end{equation}
where $a_i \in R, b_i \in G~\text{and}~c_i \in B~\text{for}~i=1,\ldots,d.$
Restriction to $H$-invariant tensors gives an isomorphism
\begin{equation}\label{eq:2}
\psi:(R^{\otimes d})^{\OO(R)}\otimes (G^{\otimes d})^{\OO(G)}\otimes (B^{\otimes d})^{\OO(B)}\to (V^{\otimes d})^H.
\end{equation}
It follows from Theorem \ref{thm:FFT} that the space $(R^{\otimes d})^{\OO(R)}$ of invariant tensors is spanned by tensors corresponding to Brauer diagrams on the set $1,2,\dots,d.$ 
We will use red edges for these diagrams. The space $(G^{\otimes d})^{\OO(G)}$ is spanned by tensors corresponding to green Brauer diagrams on the set $1,2,\dots,d$
and the space  $(B^{\otimes d})^{\OO(B)}$ is spanned by tensors corresponding to blue Brauer diagrams on $1,2,\dots,d$.
Using the isomorphism (\ref{eq:2}), we see that $(V^{\otimes d})^H$ is spanned by diagrams with vertices $1,2,3,\dots,d$ and red, green and blue edges
such that for each of the colors we have a perfect matching. This means that each vertex has exactly one red, one green and one blue edge.

\begin{definition}
A {\bf colored Brauer diagram} of size $d=2e$ is a graph with $d$ vertices labeled $1,2,\dots,d$ and $e$ red, $e$ green and $e$ blue edges
such that for each color, the edges of that color form a perfect matching.
\end{definition}
A colored Brauer diagram ${\bf D}$ on $d$ vertices
is an overlay of a red diagram ${\bf D}_R$,
a green diagram ${\bf D}_G$ and a blue diagram ${\bf D}_B$. To a colored Brauer diagram ${\bf D}$ we can associate an invariant tensor ${\mathcal T}_{\bf D}\in (V^{\otimes d})^H$
by
\begin{equation}
    {\mathcal T}_{\bf D}=\psi\big({\mathcal T}_{{\bf D}_R}\otimes {\mathcal T}_{{\bf D}_G}\otimes {\mathcal T}_{{\bf D}_B}\big).
\end{equation}
\begin{proposition}
The space $(V^{\otimes d})^H$ is spanned by all ${\mathcal T}_{\bf D}$, where ${\bf D}$ is a colored Brauer diagram on $d$ vertices.
\end{proposition}
\begin{proof}
It follows from Theorem \ref{thm:FFT} that the space $(R^{\otimes d})^{\OO(R)}$
is spanned by tensors ${\mathcal{T}}_{{\bf D}_R}$
where ${\bf D}_R$ is a red Brauer diagram on $d$ vertices.
Similarly, $(G^{\otimes d})^{\OO(G)}$ is spanned by
tensors ${\mathcal T}_{{\bf D}_G}$ and 
$(B^{\otimes d})^{\OO(B)}$ is spanned by
tensors ${\mathcal T}_{{\bf D}_B}$ where ${\bf D}_G$ and ${\bf D}_B$
are green and blue Brauer diagrams on $d$ vertices respectively.
The space 
$(R^{\otimes d})^{\OO(R)}\otimes (G^{\otimes d})^{\OO(G)}\otimes (B^{\otimes d})^{\OO(B)}$
is spanned by tensors of the form ${\mathcal T}_{{\bf D}_R}\otimes {\mathcal T}_{{\bf D}_G}\otimes {\mathcal T}_{{\bf D}_B}$. Via the isomorphism $\psi$ in (\ref{eq:2}), $(V^{\otimes d})^{H}$ is spanned by all ${\mathcal T}_{{\bf D}}=\psi({\mathcal T}_{{\bf D}_R}\otimes {\mathcal T}_{{\bf D}_G}\otimes {\mathcal T}_{{\bf D}_B})$ where ${\bf D}$ is a colored Brauer diagram
that is the overlay of ${\bf D}_R$, ${\bf D}_G$ and
${\bf D}_B$.
\end{proof}
Using the bijections from Lemma~\ref{lem:TensorBijections}, every colored Brauer diagram ${\bf D}$ on $d$ vertices corresponds to a linear map ${\mathcal L}_{\bf D}:V^{\otimes d}\to \R$ given by ${\mathcal L}_{\bf D}({\mathcal A})={\mathcal T}_{\bf D}\cdot {\mathcal A}$ for all ${\mathcal A}\in H.$
And this linear map corresponds to a multilinear map $V^d\to \R$ defined by ${\mathcal M}_{\bf D}({\mathcal A}_1,{\mathcal A}_2,\cdots,{\mathcal A}_d)={\mathcal L}_{\bf D}({\mathcal A}_1\otimes {\mathcal A}_2\otimes \cdots \otimes {\mathcal A}_d)$
for all tensors ${\mathcal A}_1,{\mathcal A}_2,\dots, {\mathcal A}_d\in V=R\otimes G\otimes B$.

\begin{corollary}\label{cor:InvariantTensorsH}
There are bijections between the following sets:
\begin{enumerate}
    \item $(V^{\otimes d})^{H}$;
    \item the set of $H$-invariant linear maps $V^{\otimes d}\to \R$;
    \item the set of $H$-invariant multilinear maps $V^d\to \R$.
\end{enumerate}
\end{corollary}
\begin{proof}
The proof is the same as for Corollary~\ref{cor:InvariantTensors}, but with $\OO(V)$ replaced by $H$.
\end{proof}

For example,
the colored Brauer diagram
{\bf D}: 
\begin{equation}\label{eq:3}
\vcenter{\xymatrix@=2.9ex{
1 \ar@[red]@{-}[r]\ar@{-}@[green][rd]  \ar@{-}@[blue][d]  & 2\ar@{-}@[green][ld]   \\
3 & 4 \ar@[red]@{-}[l] \ar@{-}@[blue][u] }}
\end{equation}
corresponds to the $H$-invariant linear map ${\mathcal L}_{\bf D}:V^{\otimes 4}=(R\otimes G\otimes B)^{\otimes 4}\to \R$ defined by
\begin{multline}\label{eq:4}
(a_1\otimes b_1\otimes c_1)\otimes  (a_2\otimes b_2\otimes c_2)\otimes (a_3\otimes b_3\otimes c_3)\otimes (a_4\otimes b_4\otimes c_4)\mapsto\\
(a_1\cdot a_2)(a_3\cdot a_4)(b_1\cdot b_4)(b_2\cdot b_3)(c_1\cdot c_3)(c_2\cdot c_4).
\end{multline}
In a similar way, we define an $H$-invariant multilinear map ${\mathcal M}_{\bf D}:V^{\otimes d}\to \R$ for every colored Brauer diagram {\bf D} of size $d$ (i.e., with $d$ vertices). We can view ${\mathcal M}_{\bf D}$ as a tensor in $(V^\star)^{\otimes d}\cong V^{\otimes d}$. 
Viewed as a tensor in $V^{\otimes d}$ we will denote it by ${\mathcal T}_{\bf D}$.

If ${\bf D}$ is a colored Brauer diagram of size $d$, then the polynomial function defined on $V$ by
\begin{equation}\label{eq:TMD}
{\mathcal T}\mapsto {\mathcal M}_{\bf D}(\underbrace{{\mathcal T},{\mathcal T},\dots,{\mathcal T}}_d)
\end{equation}
will be denoted by ${\mathcal P}_{\bf D}({\mathcal T})$. The function ${\mathcal P}_{\bf D}$ is an $H$-invariant polynomial function on $V$ of degree $d$.
Note that ${\mathcal P}_{\bf D}$ does not depend on the labeling of the vertices of ${\bf D}$. For example, if we remove the labeling
from  the diagram {\bf D} in (\ref{eq:3}) we get an unlabeled diagram 
\begin{equation}
\overline{\bf D}: \vcenter{\Xdiagram}
\end{equation}
and we define ${\mathcal P}_{\overline{\bf D}}={\mathcal P}_{\bf D}$. In coordinates, if we write ${\mathcal T}=\sum_{i=1}^p\sum_{j=1}^q\sum_{k=1}^r t_{ijk}\e_i\otimes \e_j\otimes \e_k\in R\otimes G\otimes B$, then we have
\begin{equation}\label{eq:pd_def}
{\mathcal  P}_{\bf D}({\mathcal T})=\sum_{a=1}^p\sum_{b=1}^p\sum_{c=1}^q\sum_{d=1}^q\sum_{e=1}^r\sum_{f=1}^r t_{ace}t_{adf}t_{bde}t_{bcf}.
\end{equation}
\begin{proposition}\label{prop:Hinvariants}
The space of $H$-invariant polynomial functions on $V=R\otimes G\otimes B$ is spanned by all ${\mathcal P}_{\bf D}$ where ${\bf D}$ is a colored Brauer diagram.
\end{proposition}
\begin{proof}
Let $(V^{\otimes d})^\star$ be the space of multilinear maps, and $\R[V]_d$ be the space
of homogeneous polynomial functions on $V$ of degree $d$. We have a linear map $\gamma:(V^{\otimes d})^\star\to \R[V]_d$ defined as follows:
If ${\mathcal M}:V^d\to \R$ is multilinear, 
then ${\mathcal P}=\gamma({\mathcal M})$ is given by 
\begin{equation}
{\mathcal P}({\mathcal T})={\mathcal M}(\underbrace{{\mathcal T},{\mathcal T},\dots,{\mathcal T}}_d).
\end{equation}
for all ${\mathcal T}\in V$. For a colored Brauer diagram ${\bf D}$ we have by definition $\gamma({\mathcal M}_{\bf D})={\mathcal P}_{\bf D}$ (see~(\ref{eq:TMD})). The surjective map $\gamma$ restricts to a surjective linear map of $H$-invariant subspaces $((V^{\otimes d})^\star)^{H}\to \R[V]_d^H$, 
which is also surjective by \cite[Lemma 4.2.7]{GW}.
Since $((V^{\otimes d})^\star)^H$ is spanned by the tensors ${\mathcal M}_{\bf D}$ where ${\bf D}$
is a colored Brauer diagram, $\R[V]_d^H$ is spanned by all ${\mathcal P}_{\bf D}=\gamma({\mathcal M}_{\bf D})$.
\end{proof}

If ${\bf D}\coprod {\bf E}$ is the disjoint union of colored Brauer diagrams, then we have ${\mathcal P}_{{\bf D}\coprod {\bf E}}={\mathcal P}_{\bf D}{\mathcal P}_{\bf E}$. 
\begin{corollary}
The ring of polynomial $H$-invariant polynomial functions  on $V$ is generated by invariants of the form
 ${\mathcal P}_{\bf D}$ where ${\bf D}$ is a connected colored Brauer diagram.
 \end{corollary}
 \begin{proof}
 By Proposition~\ref{prop:Hinvariants} the space of $H$-invariant polynomials is spanned by invariants of the form ${\mathcal P}_{{\bf D}}$ where ${\bf D}$ is a colored Brauer diagram. We can write ${\bf D}={\bf D}_1\coprod {\bf D}_2\coprod \cdots\coprod {\bf D}_k$
 where ${\bf D}_i$ is a connected colored Brauer diagram for every $i$. We have
 \begin{equation}
 {\mathcal P}_{{\bf D}}= {\mathcal P}_{{\bf D}_1}
 {\mathcal P}_{{\bf D}_2}\cdots {\mathcal P}_{{\bf D}_k}.
 \end{equation}
 \end{proof}
 \begin{definition} 
We can define ${\mathcal L}_{\bf D}$,
${\mathcal T}_{\bf D}$, ${\mathcal M}_{\bf D}$ and ${\mathcal P}_{\bf D}$ when ${\bf D}$ is a linear combination of diagrams by assuming that these depend linearly on ${\bf D}$. For example, if ${\bf D}=\lambda_1{\bf D}_1+\lambda_2{\bf D}_2+\cdots+\lambda_k {\bf D}_k$, then ${\mathcal P}_{\bf D}=\lambda_1{\mathcal P}_{{\bf D}_1}+\lambda_2{\mathcal P}_{{\bf D}_2}+\cdots+\lambda_k{\mathcal P}_{{\bf D}_k}$ where $\lambda_i \in \mathbb{R}$ for $i=1, \ldots, k.$
 \end{definition}

\subsection{Generators of polynomial tensor invariants}
There is only one connected colored Brauer diagram on 2 vertices:
\begin{equation}
\Plankton .
\end{equation}
There are $4$ connected colored (unlabeled) Brauer diagrams on $4$ nodes:
\begin{equation}
\RedFrame\quad \GreenFrame\quad \BlueFrame\quad \Tetrahedron\quad 
\end{equation}
There are $11$ connected colored Brauer diagrams on $6$ nodes:
\begin{equation}
\vcenter{
\xymatrix@=2.9ex{
*{\bullet} \ar@[red]@{-}[r] \ar@[green]@<.15ex>@{-}[d]\ar@[blue]@<-.15ex>@{-}[d] & *{\bullet}\ar@[blue]@{-}[d] & *{\bullet}\ar@[blue]@<.15ex>@{-}[d] \ar@[green]@{-}[l]\\
*{\bullet} \ar@[red]@{-}[r] & *{\bullet} & *{\bullet}  \ar@[green]@{-}[l] \ar@[red]@<.15ex>@{-}[u]}}\ [3] \qquad
\vcenter{
\xymatrix@=2.9ex{
*{\bullet} \ar@[blue]@{-}[r]\ar@[green]@{-}[d]\ar@[red]@{-}[rd] & *{\bullet}\ar@[green]@{-}[d] \ar@[red]@{-}[rd] & \\
*{\bullet} \ar@[blue]@{-}[r]  \ar@[red]@{-}[rd] & *{\bullet} & *{\bullet}
\ar@<.15ex>@[green]@{-}[ld]\ar@<-.15ex>@[blue]@{-}[ld] \\
& *{\bullet} &}
} \ [3] \qquad
\vcenter{
\xymatrix@=2ex{
*{\bullet}\ar@[green]@{-}[rrr]\ar@[blue]@{-}[dd]\ar@[red]@{-}[rd] & & & *{\bullet} 
\ar@[blue]@{-}[dd]\ar@[red]@{-}[ld]\\
& *{\bullet}\ar@[blue]@{-}[r] & *{\bullet} & \\
*{\bullet}\ar@[red]@{-}[rrr] \ar@[green]@{-}[ru] & & & *{\bullet} 
\ar@[green]@{-}[lu]}
} \qquad
\vcenter{
\xymatrix@=2.9ex{
*{\bullet} \ar@[blue]@{-}[r]\ar@[green]@{-}[d]\ar@[red]@{-}[rrdd] & *{\bullet}   \ar@[green]@{-}[rd]\ar@[red]@{-}[dd] & \\
*{\bullet}\ar@[red]@{-}[rr]  \ar@[blue]@{-}[rd] & & *{\bullet}
\ar@[blue]@{-}[d] \\
& *{\bullet}\ar@[green]@{-}[r] & *{\bullet} }
}  \qquad
\vcenter{
\xymatrix@=2.9ex{
*{\bullet} \ar@<.15ex>@[blue]@{-}[r]\ar@[green]@{-}[d]\ar@<-.15ex>@[red]@{-}[r] & *{\bullet}   \ar@[green]@{-}[rd] & \\
*{\bullet}\ar@<.15ex>@[red]@{-}[rd]  \ar@<-.15ex>@[blue]@{-}[rd] & & *{\bullet}\ar@<-.15ex>@[red]@{-}[d]
\ar@[blue]@<.15ex>@{-}[d] \\
& *{\bullet}\ar@[green]@{-}[r] & *{\bullet} }
}   \ [3]
\end{equation}
Here, [3] means that by permuting the colors we get 3 pairwise nonisomorphic colored graphs. For $d=2e$ vertices, the number of connected trivalent colored graphs is given in the following table:
$$
\begin{array}{|c||c|c|c|c|c|c|c|c|c|c|}
\hline
d & 2 & 4 & 6 & 8 & 10 & 12 & 14 & 16 & 18 & 20\\ \hline
\#& 1 &4 & 11 & 60 & 318 & 2806 & 29359 & 396196 & 6231794 & 112137138\\ \hline\end{array}$$
See the Online Encyclopedia of Integer Sequences (\cite{OEIS}), sequence A002831.

\begin{example}
Consider the tensors  ${\mathcal T}_1,{\mathcal T}_2\in \R^{n^2\times n^2\times n^2}$ given by
\begin{eqnarray}
{\mathcal T}_1&=&\frac{1}{n}\sum_{i=1}^{n^2} \e_i\otimes \e_i\otimes \e_i\qquad\mbox{ and}\\
{\mathcal T}_2&=&\frac{1}{n\sqrt{n}}\sum_{i=0}^{n-1}\sum_{j=0}^{n-1}\sum_{k=0}^{n-1}\e_{ni+j+1}\otimes \e_{nj+k+1}\otimes \e_{nk+i+1}.
\end{eqnarray}
Any flattening of ${\mathcal T}_1$ and ${\mathcal T}_2$ is an $n^2\times n^4$ matrix whose singular values are equal to $\frac{1}{n}$ with multiplicity $n^2$. If every vertex in a diagram is adjacent to a double or triple edge, then the corresponding tensor invariant cannot distinguish ${\mathcal T}_1$ from ${\mathcal T}_2$. In the table below, we see that only the tetrahedron diagram can distinguish ${\mathcal T}_1$ and ${\mathcal T}_2$. This invariant captures information from the tensor that cannot be seen in any flattening.

$$
\begin{array}{|c||c|c|c|c|c|} \hline
& \Plankton & \RedFrame & \GreenFrame & \BlueFrame &\Tetrahedron\\ \hline\hline
{\mathcal T}_1 & 1 & n^{-2} & n^{-2} & n^{-2} & n^{-2} \\ \hline
{\mathcal T}_2 & 1 & n^{-2} & n^{-2} & n^{-2} & 1 \\ \hline
\end{array}
$$
\end{example}

\subsection{Complexity of Tensor Invariants}
A polynomial tensor invariant corresponding to a colored Brauer diagram can be computed
from subdiagrams. 

\begin{example} To compute

\begin{equation}\label{eq:RedFrame}
\RedFrame
\end{equation}
we could first compute the partial colored Brauer diagram
\begin{equation}
\vcenter{\xymatrix@=2.9ex{\\
*{\bullet}\ar@[red]@{-}[u]
\ar@<.15ex>@[blue]@{-}[d]
\ar@<-.15ex>@[green]@{-}[d]\\
*{\bullet} \ar@[red]@{-}[d]\\
{} }}.
\end{equation}
This partial diagram corresponds to a (symmetric) tensor in $R\otimes R$.
If ${\mathcal T}=(t_{ijk})$ then this diagram corresponds to a $p\times p$ matrix $A=(a_{ij})$
where $a_{ij}=\sum_{k=1}^q\sum_{\ell=1}^r t_{ik\ell}t_{jk\ell}$.
In practice, one can compute $A$ by first flattening ${\mathcal T}$ to a $p\times (qr)$ matrix $B$
and using $A=BB^t$,  where $B^t$ is the transpose of $B$. The space complexity of this operation is $O(pqr+p^2)$ (we just have to store the tensor ${\mathcal T}$ and the matrix $A$) and the time complexity is $O(p^2qr)$,
because for each pair $(i,j)$ with $1\leq i,j\leq p,$ we have to do $O(qr)$ multiplications and additions. Finally, we compute the invariant as an inner product:
\begin{equation}\label{eq:RedFrame2}
\RedFrame\ =\ \vcenter{\xymatrix@=2.9ex{\\
*{\bullet}\ar@[red]@{-}[u]
\ar@<.15ex>@[blue]@{-}[d]
\ar@<-.15ex>@[green]@{-}[d]\\
*{\bullet} \ar@[red]@{-}[d]\\
{} } }
\ \cdot\ 
\vcenter{\xymatrix@=2.9ex{\\
*{\bullet}\ar@[red]@{-}[u]
\ar@<.15ex>@[blue]@{-}[d]
\ar@<-.15ex>@[green]@{-}[d]\\
*{\bullet} \ar@[red]@{-}[d]\\
{} } }.
\end{equation}
The space complexity of this step is $O(p^2)$ and the time complexity is $O(p^2)$.
We conclude that the space complexity of computing (\ref{eq:RedFrame}) is $O(pqr+p^2)$
and the time complexity is $O(p^2qr)$. The theoretical time complexity bounds could be improved
if we use fast matrix multiplication (such as Strassen's algorithm).
\end{example}

\begin{example}
The invariant
\begin{equation}\label{eq:Tetrahedron}
\Tetrahedron
\end{equation}
is more difficult to compute. We first compute the $p\times p\times q\times q$ tensor ${\mathcal U}$
corresponding to the diagram
\begin{equation}
\vcenter{\xymatrix@=2.9ex{
 & *{\bullet} \ar@[red]@{-}[l]  \ar@[green]@{-}[r]\ar@[blue]@{-}[d] & \\
  & *{\bullet} \ar@[red]@{-}[l]  \ar@[green]@{-}[r] & }
  }.
\end{equation}
The space complexity of this computation is $O(p^2q^2+pqr)$ and the time complexity
is $O(p^2q^2r)$. Finally we compute
\begin{equation}
   \Tetrahedron= \vcenter{\xymatrix@=2.9ex{
 & *{\bullet} \ar@[red]@{-}[l]  \ar@[green]@{-}[r]\ar@[blue]@{-}[d] & \\
  & *{\bullet} \ar@[red]@{-}[l]  \ar@[green]@{-}[r] & }
  }\ \cdot\ 
  \vcenter{\xymatrix@=2.9ex{
 & *{\bullet} \ar@[red]@{-}[l]  \ar@[green]@{-}[rd]\ar@[blue]@{-}[d] & \\
  & *{\bullet} \ar@[red]@{-}[l]  \ar@[green]@{-}[ru] & }
  }.
\end{equation}
An explicit formula for this tensor invariant is given by 
\begin{equation}
\sum_{i=1}^p\sum_{j=1}^p\sum_{k=1}^q\sum_{\ell=1}^q {\mathcal U}_{ijk\ell}{\mathcal U}_{ij\ell k}.
\end{equation}
The space complexity of this step is $O(p^2q^2)$ and the time complexity is $O(p^2q^2)$ as well. Combining the two steps, we see that the time complexity of computing the tetrahedron invariant (\ref{eq:Tetrahedron}) is $O(p^2q^2r)$ and space complexity $O(p^2q^2+pqr)$. This is the approach we would use if $p\leq q\leq r$.
If $p\geq q\geq r$
then a more efficient algorithm is obtained by switching red and blue.
In that case we get time   $O(pq^2r^2)$.
\end{example}
\begin{example}
To compute
\begin{equation}
\vcenter{
\xymatrix@=2ex{
*{\bullet}\ar@[green]@{-}[rrr]\ar@[blue]@{-}[dd]\ar@[red]@{-}[rd] & & & *{\bullet} 
\ar@[blue]@{-}[dd]\ar@[red]@{-}[ld]\\
& *{\bullet}\ar@[blue]@{-}[r] & *{\bullet} & \\
*{\bullet}\ar@[red]@{-}[rrr] \ar@[green]@{-}[ru] & & & *{\bullet} 
\ar@[green]@{-}[lu]}
}
\end{equation}
we first compute
\begin{equation}
\vcenter{
\xymatrix@=2ex{
*{\bullet}\ar@[green]@{-}[rr]\ar@[blue]@{-}[dd]\ar@[red]@{-}[rd] & & \\
& *{\bullet}\ar@[blue]@{-}[r] &  \\
*{\bullet}\ar@[red]@{-}[rr] \ar@[green]@{-}[ru] & & }
}
\end{equation}
by contracting the $p\times p\times q\times q$ tensor ${\mathcal U}$ with ${\mathcal T}$.
This step requires $O(p^2q^2+pqr)$ in memory and $O(p^2q^2r)$ in time.
From this we compute
\begin{equation}
\vcenter{
\xymatrix@=2ex{
*{\bullet}\ar@[green]@{-}[rrr]\ar@[blue]@{-}[dd]\ar@[red]@{-}[rd] & & & *{\bullet} 
\ar@[blue]@{-}[dd]\ar@[red]@{-}[ld]\\
& *{\bullet}\ar@[blue]@{-}[r] & *{\bullet} & \\
*{\bullet}\ar@[red]@{-}[rrr] \ar@[green]@{-}[ru] & & & *{\bullet} 
\ar@[green]@{-}[lu]}
}\ =\ \vcenter{
\xymatrix@=2ex{
*{\bullet}\ar@[green]@{-}[rr]\ar@[blue]@{-}[dd]\ar@[red]@{-}[rd] & & \\
& *{\bullet}\ar@[blue]@{-}[r] &  \\
*{\bullet}\ar@[red]@{-}[rr] \ar@[green]@{-}[ru] & & }
}\ \cdot\ \vcenter{
\xymatrix@=2ex{
*{\bullet}\ar@[green]@{-}[rr]\ar@[blue]@{-}[dd]\ar@[red]@{-}[rd] & & \\
& *{\bullet}\ar@[blue]@{-}[r] &  \\
*{\bullet}\ar@[red]@{-}[rr] \ar@[green]@{-}[ru] & & }
}
.\end{equation}
\end{example}

\begin{example}
To compute the below diagram
\begin{equation}
\vcenter{
\xymatrix@=2.9ex{
*{\bullet} \ar@[blue]@{-}[r]\ar@[green]@{-}[d]\ar@[red]@{-}[rrdd] & *{\bullet}   \ar@[green]@{-}[rd]\ar@[red]@{-}[dd] & \\
*{\bullet}\ar@[red]@{-}[rr]  \ar@[blue]@{-}[rd] & & *{\bullet}
\ar@[blue]@{-}[d] \\
& *{\bullet}\ar@[green]@{-}[r] & *{\bullet} }
} 
\end{equation}
we can first compute
\begin{equation}
\vcenter{
\xymatrix@=2.9ex{
& *{\bullet}\ar@[green]@{-}[l]\ar@[blue]@{-}[d] \ar@[red]@{-}[r] & \\
& *{\bullet}\ar@[red]@{-}[r]\ar@[green]@{-}[d] &\\
& *{\bullet}\ar@[blue]@{-}[l]\ar@[red]@{-}[r]
 &}
}
\end{equation}
which costs $O(p^3qr)$ in memory and $O(p^3q^2r)$ in time, and then we have
\begin{equation}
\vcenter{
\xymatrix@=2.9ex{
*{\bullet} \ar@[blue]@{-}[r]\ar@[green]@{-}[d]\ar@[red]@{-}[rrdd] & *{\bullet}   \ar@[green]@{-}[rd]\ar@[red]@{-}[dd] & \\
*{\bullet}\ar@[red]@{-}[rr]  \ar@[blue]@{-}[rd] & & *{\bullet}
\ar@[blue]@{-}[d] \\
& *{\bullet}\ar@[green]@{-}[r] & *{\bullet} }
}\ =\ 
\vcenter{
\xymatrix@=2.9ex{
& *{\bullet}\ar@[green]@{-}[l]\ar@[blue]@{-}[d] \ar@[red]@{-}[r] & \\
& *{\bullet}\ar@[red]@{-}[r]\ar@[green]@{-}[d] &\\
& *{\bullet}\ar@[blue]@{-}[l]\ar@[red]@{-}[r]
 &}
} \ \cdot\ 
\vcenter{
\xymatrix@=2.9ex{
& *{\bullet}\ar@[green]@{-}[l]\ar@[blue]@{-}[d] \ar@[red]@{-}[rdd] & \\
& *{\bullet}\ar@[red]@{-}[r]\ar@[green]@{-}[d] &\\
& *{\bullet}\ar@[blue]@{-}[l]\ar@[red]@{-}[ruu]
 &  }  

}
.
\end{equation}
\end{example}
It becomes clear that the invariants that correspond to large diagrams can be hard to compute
because of memory and time limitations. Some tensor invariants require large tensors in
intermediate steps of the computation. There is a method to improve the memory and
time requirements with some loss of the accuracy of the result. One can use the Higher Order Singular Value
Decomposition (HOSVD) to reduce a $p\times q\times r$ tensor $\mathcal{T}$ to a core tensor ${\mathcal T}'$ of size $p'\times q'\times r'$ where $p'\leq p$,
$q'\leq q$ and $r'\leq r$.
The HOSVD is a generalization of the singular value decomposition of a matrix, see~\cite{Lathauwer}. If $r>pq$ then the $p\times q\times r$ tensor can be reduced to a $p\times q\times pq$ tensor using HOSVD without any loss at all. HOSVD is a special case of Tucker decomposition \cite{Tucker}. Details of these decomposition methods are beyond the scope of this paper. 

\section{Approximations of the Spectral Norm}\label{sec4}
\subsection{The spectral norm}\label{sec4.1}
The spectral norm $\|{\mathcal T}\|_{\sigma}$ of a tensor ${\mathcal T}\in V=R\otimes G\otimes B$ is defined by
$\|{\mathcal T}\|_\sigma:=\max\{|{\mathcal T}\cdot (x\otimes y\otimes z)|\ \mid  \|x\|=\|y\|=\|z\|=1\}$.
We can view ${\mathcal T}$ as an $\ell_{\infty}$ norm on the product of unit spheres $S^{p-1}\times S^{q-1}\times S^{r-1}$. The $\ell_\infty$ norm is a limit of $\ell_d$-norms where $d\to \infty$.
We have $\|{\mathcal T}\|_\sigma:=\lim_{d\to\infty}\|{\mathcal T}\|_{\overline{\sigma},d}$
where
\begin{equation}
\|{\mathcal T}\|_{\overline{\sigma},d}:=
  \left( \int_{S^{p-1}\times S^{q-1}\times S^{r-1}} |{\mathcal T}\cdot (x\otimes y\otimes z)|^d\,d\mu\right)^{1/d}.
\end{equation}
Suppose that $d=2e$ is even. We have
$|{\mathcal T}\cdot (x\otimes y\otimes z)|^d={\mathcal T}^{\otimes d}\cdot (x\otimes y\otimes z)^{\otimes d}$
 and
 \begin{equation}
  \int_{S^{p-1}\times S^{q-1}\times S^{r-1}} |{\mathcal T}\cdot (x\otimes y\otimes z)|^d\,d\mu={\mathcal T}^{\otimes d}\cdot
   \int_{S^{p-1}\times S^{q-1}\times S^{r-1}} (x\otimes y\otimes z)^{\otimes d}\,d\mu.
 \end{equation}
Up to permutation of the tensor factors, we have the following equality
\begin{multline}
  \int_{S^{p-1}\times S^{q-1}\times S^{r-1}} (x\otimes y\otimes z)^{\otimes d}\,d\mu=  \int_{S^{p-1}\times S^{q-1}\times S^{r-1}} (x^{\otimes d}\otimes y^{\otimes d}\otimes z^{\otimes d})\,d\mu=\\=
  \left(\int_{S^{p-1}} x^{\otimes d}\,d\mu\right) \otimes \left( \int_{S^{q-1}} y^{\otimes d}\,d\mu\right) \otimes \left(\int_{S^{r-1}} z^{\otimes d}\,d\mu\right).
  \end{multline}
  We will normalize the norm $\|\cdot\|_{\overline{\sigma},d}$ so that the value of simple tensors of unit length is equal to $1$. So we define a norm
  $\|\cdot\|_{\sigma,d}$ by 
  $$\|{\mathcal T}\|_{\sigma,d}=\frac{\|{\mathcal T}\|_{\overline{\sigma},d}}{\|x\otimes y\otimes z\|_{\overline{\sigma},d}},
  $$
  where $x,y,z$ are unit vectors. We have $\lim_{d\to\infty}\|{\mathcal T}\|_{\sigma,d}=\|{\mathcal T}\|_{\sigma}$.
 We will compute $\|{\mathcal T}\|_{\sigma,d}$ for $d=2$ and $d=4$. 
  
  For any even $d$, we let ${\mathcal S}_{R,d}\in R^{\otimes d}$ be the sum of all red Brauer diagrams on $d$ vertices. For example,
  \begin{equation}
{\mathcal S}_{R,4}=\IIred+\Xred+\Zred .
\end{equation}
      Similarly ${\mathcal S}_{G,d}$ and ${\mathcal S}_{B,d}$ are the respective sums of all green and blue Brauer diagrams on $d$ nodes.
  \subsection{The approximation for \boldmath{$d=2$}}\label{sec4.2}
  \begin{proposition}
  The norm $\|\cdot\|_{\sigma,2}$ is equal to the Euclidean (or Frobenius) norm $\|\cdot\|$.
  \end{proposition}
  \begin{proof}
   If we let $e=1$, then it follows from Proposition \ref{integral} that
  \begin{equation}
  \int_{S^{p-1}}x\otimes x\,d\mu={\textstyle \frac{1}{p}}{\mathcal S}_{R,2},\quad  \int_{S^{q-1}}y\otimes y\,d\mu={\textstyle\frac{1}{q}}{\mathcal S}_{G,2},
 \mbox{ and }  \int_{S^{r-1}}z\otimes z\,d\mu={\textstyle\frac{1}{r}}{\mathcal S}_{B,2}.
  \end{equation}
  Therefore, we get
  \begin{equation}
    \int_{S^{p-1}\times S^{q-1}\times S^{r-1}} (x\otimes y\otimes z)^{\otimes 2}\,d\mu={\textstyle\frac{1}{pqr}}{\mathcal S}_{R,2}\otimes {\mathcal S}_{G,2}\otimes {\mathcal S}_{B,2}.    
    \end{equation}
  In diagrams, we get
  $$
    \int_{S^{p-1}\times S^{q-1}\times S^{r-1}} (x\otimes y\otimes z)^{\otimes 2}\,d\mu={\textstyle\frac{1}{pqr}}\ \ \Plankton .
  $$
  So we have
  $$
  \|{\mathcal T}\|_{\overline{\sigma},2}^2=(\mathcal{T}\otimes {\mathcal T})\cdot \left( {\textstyle\frac{1}{pqr}}\ \Plankton\right)={\textstyle \frac{1}{pqr}}{\mathcal T}\cdot {\mathcal T}={\textstyle \frac{1}{pqr}}\|{\mathcal T}\|^2.
  $$
  and $\|{\mathcal T}\|_{\overline{\sigma},2}=\frac{1}{\sqrt{pqr}}\|{\mathcal T}\|$.
  It follows that $\|{\mathcal T}\|_{\sigma,2}$ is equal to the Euclidean norm $\|\cdot\|$.
  
  \end{proof}
  
  \subsection{The approximation for \boldmath{$d=4$}}\label{sec4.3}
  \begin{theorem}\label{theo:approxd4}
  We have that $\|{\mathcal T}\|_{\sigma,4}^4={\mathcal P}_{\bf D}({\mathcal T})$ where 
\begin{equation}\label{eq:sigma4}
{\bf D}=
\frac{3\  \vcenter{\xymatrix@=2.9ex{
*{\bullet}\ar@<-.2ex>@[red]@{-}[d] \ar@<.2ex>@[green]@{-}[d] \ar@[blue]@{-}[d] & *{\bullet}\ar@<.2ex>@[red]@{-}[d] \ar@<-.2ex>@[green]@{-}[d]  \ar@[blue]@{-}[d] \\
*{\bullet} & *{\bullet}}} +
6\ \RedFrame+6\ \GreenFrame+6\ \BlueFrame +6\ \Tetrahedron}{27}
\end{equation}

and ${\mathcal P}_{\bf D}$ is defined as in \eqref{eq:pd_def}.
  \end{theorem}
\begin{proof}
If we employ the Proposition \ref{integral} for $e=2$, then we get
  \begin{multline}
  \int_{S^{p-1}\times S^{q-1}\times S^{r-1}} (x\otimes y\otimes z)^{\otimes 4}\,d\mu=\\=
  \left(\int_{S^{p-1}} x^{\otimes 4}\,d\mu\right) \otimes \left( \int_{S^{q-1}} y^{\otimes 4}\,d\mu\right) \otimes \left(\int_{S^{r-1}} z^{\otimes 4}\,d\mu\right)=
  {\textstyle \frac{1}{p(p+2)q(q+2)r(r+2)}}{\mathcal S}_{R,4}\otimes {\mathcal S}_{G,4}\otimes {\mathcal S}_{B,4}.
  \end{multline}
  We calculate
  \begin{multline}
  {\mathcal S}_{R,4}\otimes {\mathcal S}_{G,4}=\left(\IIred+\Xred+\Zred\right)\otimes \left(\IIgreen+\Xgreen+\Zgreen\right)=\\ \\
  \begin{array}{cccccc}
\vcenter{\xymatrix@=2.9ex{
*{\bullet}\ar@<-.2ex>@[red]@{-}[d] \ar@<.2ex>@[green]@{-}[d]  & *{\bullet}\ar@<.2ex>@[red]@{-}[d] \ar@<-.2ex>@[green]@{-}[d] \\
*{\bullet} & *{\bullet}}}
& + & 
\vcenter{\xymatrix@=2.9ex{
*{\bullet}\ar@[red]@{-}[rd]  \ar@[green]@{-}[d] & *{\bullet}\ar@[red]@{-}[ld] \ar@[green]@{-}[d] \\
*{\bullet} & *{\bullet}}}
& + & 
\vcenter{\xymatrix@=2.9ex{
*{\bullet}\ar@[red]@{-}[r] \ar@[green]@{-}[d]  & *{\bullet} \ar@[green]@{-}[d]  \\
*{\bullet}\ar@[red]@{-}[r]  & *{\bullet}}}
& + \\ \\
\vcenter{\xymatrix@=2.9ex{
*{\bullet}\ar@[red]@{-}[d] \ar@[green]@{-}[rd] & *{\bullet}\ar@[red]@{-}[d] \ar@[green]@{-}[ld]\\
*{\bullet} & *{\bullet}}}
& + & 
\vcenter{\xymatrix@=2.9ex{
*{\bullet}\ar@<.2ex>@[red]@{-}[rd] \ar@<-.2ex>@[green]@{-}[rd]  & *{\bullet}\ar@<-.2ex>@[red]@{-}[ld]\ar@<.2ex>@[green]@{-}[ld]\\
*{\bullet} & *{\bullet}}}
& + & 
\vcenter{\xymatrix@=2.9ex{
*{\bullet}\ar@[red]@{-}[r]  \ar@[green]@{-}[rd]  & *{\bullet}\ar@[green]@{-}[ld] \\
*{\bullet}\ar@[red]@{-}[r]  & *{\bullet}}}
& + \\ \\
\vcenter{\xymatrix@=2.9ex{
*{\bullet}\ar@[red]@{-}[d] \ar@[green]@{-}[r] & *{\bullet}\ar@[red]@{-}[d]\\
*{\bullet} \ar@[green]@{-}[r]  & *{\bullet}}}
& + & 
\vcenter{\xymatrix@=2.9ex{
*{\bullet}\ar@[red]@{-}[rd] \ar@[green]@{-}[r]  & *{\bullet}\ar@[red]@{-}[ld]\\
*{\bullet} \ar@[green]@{-}[r]  & *{\bullet}}}
& + & 
\vcenter{\xymatrix@=2.9ex{
*{\bullet}\ar@<.2ex>@[red]@{-}[r]  \ar@<-.2ex>@[green]@{-}[r]  & *{\bullet} \\ 
*{\bullet}\ar@<-.2ex>@[red]@{-}[r] \ar@<.2ex>@[green]@{-}[r]  & *{\bullet}}}
& 
\end{array}\quad=3\ \vcenter{\xymatrix@=2.9ex{
*{\bullet}\ar@<-.2ex>@[red]@{-}[d] \ar@<.2ex>@[green]@{-}[d]  & *{\bullet}\ar@<.2ex>@[red]@{-}[d] \ar@<-.2ex>@[green]@{-}[d] \\
*{\bullet} & *{\bullet}}}+
6\ \vcenter{\xymatrix@=2.9ex{
*{\bullet}\ar@[red]@{-}[r] \ar@[green]@{-}[d]  & *{\bullet} \ar@[green]@{-}[d]  \\
*{\bullet}\ar@[red]@{-}[r]  & *{\bullet}}}.
  \end{multline}
 In this calculation, we have omitted the labeling of the vertices. The last equality is only true if we symmetrize the right-hand side over all 24 permutations on the 4 vertices. 
   \begin{multline}
  {\mathcal S}_{R,4}\otimes {\mathcal S}_{G,4}\otimes {\mathcal S}_{B,4}=
  \left(3\ \vcenter{\xymatrix@=2.9ex{
*{\bullet}\ar@<-.2ex>@[red]@{-}[d] \ar@<.2ex>@[green]@{-}[d]  & *{\bullet}\ar@<.2ex>@[red]@{-}[d] \ar@<-.2ex>@[green]@{-}[d] \\
*{\bullet} & *{\bullet}}}+
6\ \vcenter{\xymatrix@=2.9ex{
*{\bullet}\ar@[red]@{-}[r] \ar@[green]@{-}[d]  & *{\bullet} \ar@[green]@{-}[d]  \\
*{\bullet}\ar@[red]@{-}[r]  & *{\bullet}}}\right)\otimes \left(\IIblue+\Xblue+\Zblue\right)=\\ \ \\
\begin{array}{cccccc}
3\ \vcenter{\xymatrix@=2.9ex{
*{\bullet}\ar@<-.2ex>@[red]@{-}[d] \ar@<.2ex>@[green]@{-}[d] \ar@[blue]@{-}[d] & *{\bullet}\ar@<.2ex>@[red]@{-}[d] \ar@<-.2ex>@[green]@{-}[d]  \ar@[blue]@{-}[d] \\
*{\bullet} & *{\bullet}}} 
& + &
3\ \vcenter{\xymatrix@=2.9ex{
*{\bullet}\ar@<-.2ex>@[red]@{-}[d] \ar@<.2ex>@[green]@{-}[d]\ar@[blue]@{-}[rd]  & *{\bullet}\ar@<.2ex>@[red]@{-}[d] \ar@<-.2ex>@[green]@{-}[d] \ar@[blue]@{-}[ld]  \\
*{\bullet} & *{\bullet}}}
& + &
3\ \vcenter{\xymatrix@=2.9ex{
*{\bullet}\ar@<-.2ex>@[red]@{-}[d] \ar@<.2ex>@[green]@{-}[d] \ar@[blue]@{-}[r] & *{\bullet}\ar@<.2ex>@[red]@{-}[d] \ar@<-.2ex>@[green]@{-}[d] \\
*{\bullet}  \ar@[blue]@{-}[r] & *{\bullet}}}
& +\\ \\
6\ \vcenter{\xymatrix@=2.9ex{
*{\bullet}\ar@[red]@{-}[r] \ar@<-.2ex>@[green]@{-}[d]  \ar@<.2ex>@[blue]@{-}[d]   & *{\bullet} \ar@<.2ex>@[green]@{-}[d]  \ar@<-.2ex>@[blue]@{-}[d]  \\
*{\bullet}\ar@[red]@{-}[r]  & *{\bullet}}}
& + &
6\ \vcenter{\xymatrix@=2.9ex{
*{\bullet}\ar@[red]@{-}[r] \ar@[green]@{-}[d] \ar@[blue]@{-}[rd]  & *{\bullet} \ar@[green]@{-}[d] \ar@[blue]@{-}[ld]  \\
*{\bullet}\ar@[red]@{-}[r]  & *{\bullet}}}
& + &
6\ \vcenter{\xymatrix@=2.9ex{
*{\bullet}\ar@<.2ex>@[red]@{-}[r] \ar@[green]@{-}[d]  \ar@<-.2ex>@[blue]@{-}[r] & *{\bullet} \ar@[green]@{-}[d]  \\
*{\bullet}\ar@<-.2ex>@[red]@{-}[r] \ar@<.2ex>@[blue]@{-}[r]  & *{\bullet}}}
&\end{array} =
3\  \vcenter{\xymatrix@=2.9ex{
*{\bullet}\ar@<-.2ex>@[red]@{-}[d] \ar@<.2ex>@[green]@{-}[d] \ar@[blue]@{-}[d] & *{\bullet}\ar@<.2ex>@[red]@{-}[d] \ar@<-.2ex>@[green]@{-}[d]  \ar@[blue]@{-}[d] \\
*{\bullet} & *{\bullet}}} +
6\ \RedFrame+6 \GreenFrame+6\ \BlueFrame +6\ \Tetrahedron .
\end{multline}
For this calculation, one should symmetrize the red-green diagrams over all 24 permutations. However, if we do not do this the result will not change
because the blue diagrams are symmetrized over all permutations. We conclude that $\|{\mathcal T}\|_{\sigma,4}^4={\mathcal P}_{\bf D}({\mathcal T})$ where 
\begin{equation}
{\bf D}=
\frac{3\  \vcenter{\xymatrix@=2.9ex{
*{\bullet}\ar@<-.2ex>@[red]@{-}[d] \ar@<.2ex>@[green]@{-}[d] \ar@[blue]@{-}[d] & *{\bullet}\ar@<.2ex>@[red]@{-}[d] \ar@<-.2ex>@[green]@{-}[d]  \ar@[blue]@{-}[d] \\
*{\bullet} & *{\bullet}}} +
6\ \RedFrame+6\ \GreenFrame+6\ \BlueFrame +6\ \Tetrahedron}{27}.
\end{equation}
\end{proof}
\subsection{Other approximations of the spectral norm}\label{sec4.4}
We say that a norm $\|\cdot\|_{\#}$ is a degree $d$ norm if $\|{\mathcal T}\|_{\#}^d$ is a polynomial function on ${\mathcal T}$ of degree $d$.
The norm $\|\cdot\|_{\sigma,d}$ is a norm of degree $d$. In particular, $\|\cdot\|_{\sigma,4}$ is a norm of degree $4$. In this section
we study other norms of degree $4$ that approximate the spectral norm.

Consider the degree 2 covariant ${\mathcal U}=V\to G^{\otimes 2}\otimes B^{\otimes 2}$ defined by
$$
{\mathcal U}={\mathcal U}({\mathcal T})=
 \vcenter{\xymatrix@=2.9ex{
*{\bullet}\ar@[green]@{-}[r] \ar@[red]@{-}[d] \ar@[blue]@{-}[rd]  &  \ar@[blue]@{-}[ld]  \\
*{\bullet}\ar@[green]@{-}[r]  &  }} + 
 \vcenter{\xymatrix@=2.9ex{
*{\bullet}\ar@<.2ex>@[green]@{-}[r] \ar@[red]@{-}[d] \ar@<-.2ex>@[blue]@{-}[r]  &   \\
*{\bullet}\ar@<-.2ex>@[green]@{-}[r]  &   \ar@<-.2ex>@[blue]@{-}[l] }} 
.$$
We have
\begin{equation}\label{eq:matrixdot}
0\leq 
 \vcenter{\xymatrix@=2.9ex{
*{\bullet}\ar@[green]@{-}[r] \ar@[red]@{-}[d] \ar@[blue]@{-}[rd]  &  \ar@[blue]@{-}[ld]  \\
*{\bullet}\ar@[green]@{-}[r]  &  }} \cdot  \vcenter{\xymatrix@=2.9ex{
*{\bullet}\ar@[green]@{-}[r] \ar@[red]@{-}[d] \ar@[blue]@{-}[rd]  &  \ar@[blue]@{-}[ld]  \\
*{\bullet}\ar@[green]@{-}[r]  &  }}=
 \vcenter{\xymatrix@=2.9ex{
*{\bullet}\ar@<.2ex>@[green]@{-}[r] \ar@[red]@{-}[d] \ar@<-.2ex>@[blue]@{-}[r]  &   \\
*{\bullet}\ar@<-.2ex>@[green]@{-}[r]  &   \ar@<-.2ex>@[blue]@{-}[l] }} 
\cdot
 \vcenter{\xymatrix@=2.9ex{
*{\bullet}\ar@<.2ex>@[green]@{-}[r] \ar@[red]@{-}[d] \ar@<-.2ex>@[blue]@{-}[r]  &   \\
*{\bullet}\ar@<-.2ex>@[green]@{-}[r]  &   \ar@<-.2ex>@[blue]@{-}[l] }} =
 \vcenter{\xymatrix@=2.9ex{
*{\bullet}\ar@<.2ex>@[green]@{-}[r] \ar@[red]@{-}[d] \ar@<-.2ex>@[blue]@{-}[r]  & *{\bullet}\ar@[red]@{-}[d]  \\
*{\bullet}\ar@<-.2ex>@[green]@{-}[r]  &  *{\bullet} \ar@<-.2ex>@[blue]@{-}[l] }} \qquad\mbox{and}
\end{equation}
\begin{equation}
 \vcenter{\xymatrix@=2.9ex{
*{\bullet}\ar@[green]@{-}[r] \ar@[red]@{-}[d] \ar@[blue]@{-}[rd]  &  \ar@[blue]@{-}[ld]  \\
*{\bullet}\ar@[green]@{-}[r]  &  }} \cdot 
 \vcenter{\xymatrix@=2.9ex{
*{\bullet}\ar@<.2ex>@[green]@{-}[r] \ar@[red]@{-}[d] \ar@<-.2ex>@[blue]@{-}[r]  &   \\
*{\bullet}\ar@<-.2ex>@[green]@{-}[r]  &   \ar@<-.2ex>@[blue]@{-}[l] }}  =
 \vcenter{\xymatrix@=2.9ex{
*{\bullet}\ar@[green]@{-}[r] \ar@[red]@{-}[d] \ar@[blue]@{-}[rd]  &  *{\bullet}  \ar@[blue]@{-}[ld] \ar@[red]@{-}[d] \\
*{\bullet}\ar@[green]@{-}[r]  &  *{\bullet} }}
\end{equation}
(in this calculation, the diagrams represent their evaluations on ${\mathcal T}\otimes{\mathcal T}\otimes {\mathcal T}\otimes{\mathcal T}).$ So we get
\begin{equation}\label{eq:ineq1}
 \vcenter{\xymatrix@=2.9ex{
*{\bullet}\ar@<.2ex>@[green]@{-}[r] \ar@[red]@{-}[d] \ar@<-.2ex>@[blue]@{-}[r]  & *{\bullet}\ar@[red]@{-}[d]  \\
*{\bullet}\ar@<-.2ex>@[green]@{-}[r]  &  *{\bullet} \ar@<-.2ex>@[blue]@{-}[l] }} +
 \vcenter{\xymatrix@=2.9ex{
*{\bullet}\ar@[green]@{-}[r] \ar@[red]@{-}[d] \ar@[blue]@{-}[rd]  &  *{\bullet}  \ar@[blue]@{-}[ld] \ar@[red]@{-}[d] \\
*{\bullet}\ar@[green]@{-}[r]  &  *{\bullet} }}={\textstyle \frac{1}{2}}({\mathcal U}\cdot {\mathcal U})\geq 0.
\end{equation}
Permuting the colors also gives
\begin{equation}\label{eq:ineq2}
 \vcenter{\xymatrix@=2.9ex{
*{\bullet}\ar@<.2ex>@[red]@{-}[r] \ar@[green]@{-}[d] \ar@<-.2ex>@[blue]@{-}[r]  & *{\bullet}\ar@[green]@{-}[d]  \\
*{\bullet}\ar@<-.2ex>@[red]@{-}[r]  &  *{\bullet} \ar@<-.2ex>@[blue]@{-}[l] }} +
 \vcenter{\xymatrix@=2.9ex{
*{\bullet}\ar@[green]@{-}[r] \ar@[red]@{-}[d] \ar@[blue]@{-}[rd]  &  *{\bullet}  \ar@[blue]@{-}[ld] \ar@[red]@{-}[d] \\
*{\bullet}\ar@[green]@{-}[r]  &  *{\bullet} }}\geq 0.
\end{equation}
It follows from (\ref{eq:matrixdot}) that
\begin{equation}\label{eq:ineq3}
 \vcenter{\xymatrix@=2.9ex{
*{\bullet}\ar@<.2ex>@[red]@{-}[r] \ar@[blue]@{-}[d] \ar@<-.2ex>@[green]@{-}[r]  & *{\bullet}\ar@[blue]@{-}[d]  \\
*{\bullet}\ar@<-.2ex>@[red]@{-}[r]  &  *{\bullet} \ar@<-.2ex>@[green]@{-}[l] }} \geq0.
\end{equation}
Adding (\ref{eq:ineq1}), (\ref{eq:ineq2}) and (\ref{eq:ineq3}) gives
\begin{equation}\label{eq:NormNonnegative}
 \vcenter{\xymatrix@=2.9ex{
*{\bullet}\ar@<.2ex>@[red]@{-}[r] \ar@[blue]@{-}[d] \ar@<-.2ex>@[green]@{-}[r]  & *{\bullet}\ar@[blue]@{-}[d]  \\
*{\bullet}\ar@<-.2ex>@[red]@{-}[r]  &  *{\bullet} \ar@<-.2ex>@[green]@{-}[l] }}+
 \vcenter{\xymatrix@=2.9ex{
*{\bullet}\ar@<.2ex>@[red]@{-}[r] \ar@[green]@{-}[d] \ar@<-.2ex>@[blue]@{-}[r]  & *{\bullet}\ar@[green]@{-}[d]  \\
*{\bullet}\ar@<-.2ex>@[red]@{-}[r]  &  *{\bullet} \ar@<-.2ex>@[blue]@{-}[l] }} +
 \vcenter{\xymatrix@=2.9ex{
*{\bullet}\ar@<.2ex>@[green]@{-}[r] \ar@[red]@{-}[d] \ar@<-.2ex>@[blue]@{-}[r]  & *{\bullet}\ar@[red]@{-}[d]  \\
*{\bullet}\ar@<-.2ex>@[green]@{-}[r]  &  *{\bullet} \ar@<-.2ex>@[blue]@{-}[l] }} +
2 \vcenter{\xymatrix@=2.9ex{
*{\bullet}\ar@[green]@{-}[r] \ar@[red]@{-}[d] \ar@[blue]@{-}[rd]  &  *{\bullet}  \ar@[blue]@{-}[ld] \ar@[red]@{-}[d] \\
*{\bullet}\ar@[green]@{-}[r]  &  *{\bullet} }}\geq 0.
\end{equation}
\begin{definition}\label{def:NewNorm}
We define
\begin{equation}\label{Tsharp}
\|{\mathcal T}\|_\#=\frac{1}{5^{1/4}}\left(\vcenter{\xymatrix@=2.9ex{
*{\bullet}\ar@<.2ex>@[red]@{-}[r] \ar@[blue]@{-}[d] \ar@<-.2ex>@[green]@{-}[r]  & *{\bullet}\ar@[blue]@{-}[d]  \\
*{\bullet}\ar@<-.2ex>@[red]@{-}[r]  &  *{\bullet} \ar@<-.2ex>@[green]@{-}[l] }}+
 \vcenter{\xymatrix@=2.9ex{
*{\bullet}\ar@<.2ex>@[red]@{-}[r] \ar@[green]@{-}[d] \ar@<-.2ex>@[blue]@{-}[r]  & *{\bullet}\ar@[green]@{-}[d]  \\
*{\bullet}\ar@<-.2ex>@[red]@{-}[r]  &  *{\bullet} \ar@<-.2ex>@[blue]@{-}[l] }} +
 \vcenter{\xymatrix@=2.9ex{
*{\bullet}\ar@<.2ex>@[green]@{-}[r] \ar@[red]@{-}[d] \ar@<-.2ex>@[blue]@{-}[r]  & *{\bullet}\ar@[red]@{-}[d]  \\
*{\bullet}\ar@<-.2ex>@[green]@{-}[r]  &  *{\bullet} \ar@<-.2ex>@[blue]@{-}[l] }} +
2 \vcenter{\xymatrix@=2.9ex{
*{\bullet}\ar@[green]@{-}[r] \ar@[red]@{-}[d] \ar@[blue]@{-}[rd]  &  *{\bullet}  \ar@[blue]@{-}[ld] \ar@[red]@{-}[d] \\
*{\bullet}\ar@[green]@{-}[r]  &  *{\bullet} }}\right)^{1/4}.
\end{equation}
\end{definition}
We will show that $\|{\mathcal T}\|_\#$ is a norm.
\begin{lemma}
Suppose that $f(x)=f(x_1,x_2,\dots,x_m)\in \R[x_1,\dots,x_m]$ is a homogeneous polynomial of degree $d>0$ with $f(x)>0$ 
for all nonzero $x\in \R^m$ and the Hessian matrix $(\frac{\partial^2 f}{\partial x_i\partial x_j})$
is positive semi-definite.  Then $\|x\|_\#:=f(x)^{1/d}$ is a norm on $\R^m$.
\end{lemma}
\begin{proof}
It is clear that $\|x\|_\#=0$ if and only if $x=0$. We have $f(\lambda x)=\lambda^d f(x)$ which implies that $d$ must be even.
 We get  $f(\lambda x)^{1/d}=(\lambda^d f(x))^{1/d}=|\lambda| f(x)^{1/d}$.
Because the Hessian is positive semi-definite, the function $f(x)$ is convex and the set $B=\{x\mid f(x)\leq 1\}$ is convex,
which is also the unit ball for $\|x\|_\#$.

If $x,y\in \R^n$ are nonzero, then we have
$\frac{x}{\|x\|_\#},\frac{y}{\|y\|_\#}\in B$ and therefore 
$$\frac{x+y}{\|x\|_\#+\|y\|_\#}=\frac{\|x\|_\#}{\|x\|_\#+\|y\|_\#}\cdot \frac{x}{\|x\|_\#}+\frac{\|y\|_\#}{\|x\|_\#+\|y\|_\#}\cdot \frac{y}{\|y\|_\#}\in B.
$$
So 
$$
\left\|\frac{x+y}{\|x\|_\#+\|y\|_\#}\right\|_\#=\frac{\|x+y\|_\#}{\|x\|_\#+\|y\|_\#}\leq 1.
$$
This proves the triangle inequality.
\end{proof}

\begin{proposition}\label{prop:IsNorm}
The function $\|\cdot\|_\#$ is a norm.
\end{proposition}
\begin{proof}
From (\ref{eq:NormNonnegative}) it follows that $\|\cdot\|_\#$ is nonnegative. If $\|{\mathcal T}\|_\#=0$ for some tensor, then we have equality in (\ref{eq:NormNonnegative}), (\ref{eq:ineq3})
and (\ref{eq:matrixdot}). This implies that
$$
 \vcenter{\xymatrix@=2.9ex{
*{\bullet}\ar@<.2ex>@[green]@{-}[r] \ar@[red]@{-}[d] \ar@<-.2ex>@[blue]@{-}[r]  &   \\
*{\bullet}\ar@<-.2ex>@[green]@{-}[r]  &   \ar@<-.2ex>@[blue]@{-}[l] }}=0. 
$$
If $A$ is a $p\times qr$ flattening of ${\mathcal T}$, then we have $A^t A=0$ where $A^t$ is the transpose of $A$.
It follows that $A=0$ and ${\mathcal T}=0$. 
To show that $\|\cdot\|_\#$ satisfies the triangle inequality, we have to show that the Hessian of $h=\|\cdot\|_\#^4$ is nonnegative.

Up to a constant, $h$ is equal to  (\ref{eq:NormNonnegative}).
We can write
$$
h(\mathcal {T+E})=h_0(\mathcal {T,E})+h_1(\mathcal {T,E})+h_2(\mathcal{T,E})+h_3(\mathcal {T,E})+h_4(\mathcal {T,E})
$$
where $h_i(\mathcal {T,E})$ is a polynomial function of degree $4-i$ in $\mathcal{T}$ and degree $i$ in $E$.
Here $h_0(\mathcal{T,E})=\|\mathcal{T}\|_\#^4$ and  $h_4(\mathcal{T,E})=\|\mathcal{E}\|_\#^4$. The function $h_1(\mathcal{T,E})$ 
is linear in $\mathcal{E}$ and this linear function is the gradient at ${\mathcal T}$. 
The function $h_2(\mathcal{T,E})$ is quadratic function in $\mathcal{E}$ and is, up to a constant, the Hessian
of $h$ at $\mathcal{T}$. So we have to show that $h_2(\mathcal{T,E})\geq 0$ for all tensors $\mathcal{T}$ and $\mathcal{E}$.

Let us write a black vertex for the tensor $\mathcal{T}$ and a white vertex for the tensor $\mathcal{E}$.
We get the Hessian of a function in $\mathcal{T}$ by summing all the possible ways of replacing two
black vertices by two white vertices. The Hessian of the left-hand side of (\ref{eq:NormNonnegative}) is
\begin{equation}\label{eq:h2}
{\textstyle \frac{1}{2}}h_2(\mathcal{T,E})=
\begin{array}{cccccccc}
 \vcenter{\xymatrix@=2.9ex{
*{\circ}\ar@<.2ex>@[red]@{-}[r] \ar@[blue]@{-}[d] \ar@<-.2ex>@[green]@{-}[r]  & *{\circ}\ar@[blue]@{-}[d]  \\
*{\bullet}\ar@<-.2ex>@[red]@{-}[r]  &  *{\bullet} \ar@<-.2ex>@[green]@{-}[l] }}
&+&
 \vcenter{\xymatrix@=2.9ex{
*{\circ}\ar@<.2ex>@[red]@{-}[r] \ar@[green]@{-}[d] \ar@<-.2ex>@[blue]@{-}[r]  & *{\circ}\ar@[green]@{-}[d]  \\
*{\bullet}\ar@<-.2ex>@[red]@{-}[r]  &  *{\bullet} \ar@<-.2ex>@[blue]@{-}[l] }} 
&+&
 \vcenter{\xymatrix@=2.9ex{
*{\circ}\ar@<.2ex>@[green]@{-}[r] \ar@[red]@{-}[d] \ar@<-.2ex>@[blue]@{-}[r]  & *{\circ}\ar@[red]@{-}[d]  \\
*{\bullet}\ar@<-.2ex>@[green]@{-}[r]  &  *{\bullet} \ar@<-.2ex>@[blue]@{-}[l] }}
&+&
2 \vcenter{\xymatrix@=2.9ex{
*{\circ}\ar@[green]@{-}[r] \ar@[red]@{-}[d] \ar@[blue]@{-}[rd]  &  *{\circ}  \ar@[blue]@{-}[ld] \ar@[red]@{-}[d] \\
*{\bullet}\ar@[green]@{-}[r]  &  *{\bullet} }}
&+ \\ \\
 \vcenter{\xymatrix@=2.9ex{
*{\circ}\ar@<.2ex>@[red]@{-}[r] \ar@[blue]@{-}[d] \ar@<-.2ex>@[green]@{-}[r]  & *{\bullet}\ar@[blue]@{-}[d]  \\
*{\circ}\ar@<-.2ex>@[red]@{-}[r]  &  *{\bullet} \ar@<-.2ex>@[green]@{-}[l] }}
&+&
 \vcenter{\xymatrix@=2.9ex{
*{\circ}\ar@<.2ex>@[red]@{-}[r] \ar@[green]@{-}[d] \ar@<-.2ex>@[blue]@{-}[r]  & *{\bullet}\ar@[green]@{-}[d]  \\
*{\circ}\ar@<-.2ex>@[red]@{-}[r]  &  *{\bullet} \ar@<-.2ex>@[blue]@{-}[l] }} 
&+&
 \vcenter{\xymatrix@=2.9ex{
*{\circ}\ar@<.2ex>@[green]@{-}[r] \ar@[red]@{-}[d] \ar@<-.2ex>@[blue]@{-}[r]  & *{\bullet}\ar@[red]@{-}[d]  \\
*{\circ}\ar@<-.2ex>@[green]@{-}[r]  &  *{\bullet} \ar@<-.2ex>@[blue]@{-}[l] }}
&+&
2 \vcenter{\xymatrix@=2.9ex{
*{\circ}\ar@[green]@{-}[r] \ar@[red]@{-}[d] \ar@[blue]@{-}[rd]  &  *{\bullet}  \ar@[blue]@{-}[ld] \ar@[red]@{-}[d] \\
*{\circ}\ar@[green]@{-}[r]  &  *{\bullet} }}
&+ \\ \\
 \vcenter{\xymatrix@=2.9ex{
*{\circ}\ar@<.2ex>@[red]@{-}[r] \ar@[blue]@{-}[d] \ar@<-.2ex>@[green]@{-}[r]  & *{\bullet}\ar@[blue]@{-}[d]  \\
*{\bullet}\ar@<-.2ex>@[red]@{-}[r]  &  *{\circ} \ar@<-.2ex>@[green]@{-}[l] }}
&+&
 \vcenter{\xymatrix@=2.9ex{
*{\circ}\ar@<.2ex>@[red]@{-}[r] \ar@[green]@{-}[d] \ar@<-.2ex>@[blue]@{-}[r]  & *{\bullet}\ar@[green]@{-}[d]  \\
*{\bullet}\ar@<-.2ex>@[red]@{-}[r]  &  *{\circ} \ar@<-.2ex>@[blue]@{-}[l] }} 
&+&
 \vcenter{\xymatrix@=2.9ex{
*{\circ}\ar@<.2ex>@[green]@{-}[r] \ar@[red]@{-}[d] \ar@<-.2ex>@[blue]@{-}[r]  & *{\bullet}\ar@[red]@{-}[d]  \\
*{\bullet}\ar@<-.2ex>@[green]@{-}[r]  &  *{\circ} \ar@<-.2ex>@[blue]@{-}[l] }}
&+&
2 \vcenter{\xymatrix@=2.9ex{
*{\circ}\ar@[green]@{-}[r] \ar@[red]@{-}[d] \ar@[blue]@{-}[rd]  &  *{\bullet}  \ar@[blue]@{-}[ld] \ar@[red]@{-}[d] \\
*{\bullet}\ar@[green]@{-}[r]  &  *{\circ} }}.
\end{array}
\end{equation}
Let ${\mathcal W}:V\otimes V\to G^{\otimes 2}\otimes B^{\otimes 2}$ be defined by
$$
{\mathcal W}(\mathcal{T,E})=
 \vcenter{\xymatrix@=2.9ex{
*{\bullet}\ar@[green]@{-}[r] \ar@[red]@{-}[d] \ar@[blue]@{-}[rd]  &  \ar@[blue]@{-}[ld]  \\
*{\circ}\ar@[green]@{-}[r]  &  }} + 
\vcenter{\xymatrix@=2.9ex{
*{\circ}\ar@[green]@{-}[r] \ar@[red]@{-}[d] \ar@[blue]@{-}[rd]  &  \ar@[blue]@{-}[ld]  \\
*{\bullet}\ar@[green]@{-}[r]  &  }} + 
 \vcenter{\xymatrix@=2.9ex{
*{\bullet}\ar@<.2ex>@[green]@{-}[r] \ar@[red]@{-}[d] \ar@<-.2ex>@[blue]@{-}[r]  &   \\
*{\circ}\ar@<-.2ex>@[green]@{-}[r]  &   \ar@<-.2ex>@[blue]@{-}[l] }} +
 \vcenter{\xymatrix@=2.9ex{
*{\circ}\ar@<.2ex>@[green]@{-}[r] \ar@[red]@{-}[d] \ar@<-.2ex>@[blue]@{-}[r]  &   \\
*{\bullet}\ar@<-.2ex>@[green]@{-}[r]  &   \ar@<-.2ex>@[blue]@{-}[l] }}. 
$$
We compute
\begin{equation}\label{eq:WW}
0\leq {\textstyle\frac{1}{4}}(
{\mathcal W}\cdot {\mathcal W})=
  \vcenter{\xymatrix@=2.9ex{
*{\circ}\ar@[green]@{-}[r] \ar@[red]@{-}[d] \ar@[blue]@{-}[rd]  & *{\circ} \ar@[blue]@{-}[ld]  \ar@[red]@{-}[d] \\
*{\bullet}\ar@[green]@{-}[r]  & *{\bullet} }} +
  \vcenter{\xymatrix@=2.9ex{
*{\circ}\ar@[green]@{-}[r] \ar@[red]@{-}[d] \ar@[blue]@{-}[rd]  & *{\bullet} \ar@[blue]@{-}[ld]  \ar@[red]@{-}[d] \\
*{\bullet}\ar@[green]@{-}[r]  & *{\circ} }}+
 \vcenter{\xymatrix@=2.9ex{
*{\circ}\ar@<.2ex>@[green]@{-}[r] \ar@[red]@{-}[d] \ar@<-.2ex>@[blue]@{-}[r]  & *{\circ}\ar@[red]@{-}[d]  \\
*{\bullet}\ar@<-.2ex>@[green]@{-}[r]  & *{\bullet}  \ar@<-.2ex>@[blue]@{-}[l] }} +
 \vcenter{\xymatrix@=2.9ex{
*{\circ}\ar@<.2ex>@[green]@{-}[r] \ar@[red]@{-}[d] \ar@<-.2ex>@[blue]@{-}[r]  & *{\bullet}\ar@[red]@{-}[d]  \\
*{\bullet}\ar@<-.2ex>@[green]@{-}[r]  & *{\circ}  \ar@<-.2ex>@[blue]@{-}[l] }}.
\end{equation}
and we have
\begin{equation}\label{eq:nonneg}
0\leq 
 \vcenter{\xymatrix@=2.9ex{
*{\circ}\ar@[green]@{-}[r] \ar@<.2ex>@[red]@{-}[d] \ar@<-.2ex>@[blue]@{-}[d]  &   \\
*{\bullet}\ar@[green]@{-}[r]  &  }}\cdot
 \vcenter{\xymatrix@=2.9ex{
*{\circ}\ar@[green]@{-}[r] \ar@<.2ex>@[red]@{-}[d] \ar@<-.2ex>@[blue]@{-}[d]  &   \\
*{\bullet}\ar@[green]@{-}[r]  &  }}=
 \vcenter{\xymatrix@=2.9ex{
*{\circ}\ar@<.2ex>@[red]@{-}[r] \ar@[green]@{-}[d] \ar@<-.2ex>@[blue]@{-}[r]  & *{\bullet}\ar@[green]@{-}[d]  \\
*{\circ}\ar@<-.2ex>@[red]@{-}[r]  &  *{\bullet} \ar@<-.2ex>@[blue]@{-}[l] }}.  
\end{equation}
Adding  (\ref{eq:WW}) and (\ref{eq:nonneg}) and all expressions obtained by cyclically permuting the colors red, green and blue yields (\ref{eq:h2})
This proves that $h_2(\mathcal{T,E})\geq 0$ and completes the proof that $\|\cdot\|_\#$ is a norm.
\end{proof}

\begin{definition}
A spectral-like  norm is a norm $\|\cdot\|_X$ in $\R^{p\times q\times r}$
with the following properties:
\begin{enumerate}
    \item $\|\mathcal{T}\|_X=1$ if $\mathcal{T}$ is a rank $1$ tensor with $\|\mathcal{T}\|_2=1$;
    \item $\|\mathcal{T}\|_X<1$ if ${\mathcal T}$ is a tensor of rank $>1$ with $\|\mathcal{T}\|_2=1$.
\end{enumerate}
\end{definition}
Examples of spectral-like norms are the spectral norm $\|\cdot\|_\sigma$, the norms $\|\cdot\|_{\sigma,d}$ for $d=2,4,\dots$ and $\|\cdot\|_{\#}$.
\begin{definition}
A nuclear-like norm is a norm $\|\cdot\|_Y$ in $\R^{p\times q\times r}$
with the following properties:
\begin{enumerate}
    \item $\|{\mathcal T}\|_Y=1$ if ${\mathcal T}$ is a rank $1$ tensor with $\|{\mathcal T}\|_2=1$;
    \item $\|{\mathcal T}\|_Y>1$ if ${\mathcal T}$ is a tensor of rank $>1$ with $\|{\mathcal T}\|_2=1$.
\end{enumerate}
\end{definition}
A norm  $\|\cdot\|_Y$ is the dual of another norm $\|\cdot\|_X$ if
$$
\|{\mathcal S}\|_Y=\max\{{\mathcal S}\cdot {\mathcal T}~:~ \|{\mathcal T}\|_X\ \leq 1 \}.
$$
A norm $\|\cdot\|_Y$ is the dual of $\|\cdot\|_X$ if and only if $\|\cdot\|_X$ is the dual of $\|\cdot\|_Y$.
The dual of a spectral-like norm is a nuclear-like norm.

We are particularly interested in norms that are powerful in distinguishing
low rank tensors from high rank tensors. Spectral-like norms are normalized
such that rank 1 tensors of unit Euclidean length have norm 1.
A possible measure for the rank discriminating power of a spectral-like
norm $\|\cdot\|_X$ is the expected value of $\EE(\|{\mathcal T}\|_X)$
where ${\mathcal T}\in S^{pqr-1}$ is a random unit tensor in $R\otimes G\otimes B$ (with the uniform distribution over the sphere). A smaller value of $\EE(\|{\mathcal T}\|_X)$
means more discriminating power, which is better.
In this sense, the spectral norm is the best norm, because for spectral-like norms $\|\cdot\|_X$ we have $\|{\mathcal T}\|_X\geq \|{\mathcal T}\|_\sigma$, so $\EE(\|{\mathcal T}\|_X)\geq \EE(\|{\mathcal T}\|_\sigma)$. We may not be able to compute the value $\EE(\|{\mathcal T}\|_X)$
for many norms $\|\cdot\|_X$. If we fix the size of the tensor we can
estimate $\EE(\|{\mathcal T}\|_X)$ by taking the average of random unit vectors $x$.

We will compare the norms $\|\cdot\|_{\sigma,4}$ and $\|\cdot\|_{\#}$, which both have degree $4$. Although we are not able to give closed formulas
for $\EE(\|{\mathcal T}\|_{\sigma,4})$ and $\EE(\|{\mathcal T}\|_{\#})$, we can compute
 $\EE(\|{\mathcal T}\|_{\sigma,4}^4)$ and $\EE\left(\|{\mathcal T}\|_{\#}^4\right)$.
 First we note that
 \begin{equation}\label{eq:TTTT}
 \EE({\mathcal T}\otimes {\mathcal T} \otimes {\mathcal T}\otimes \mathcal {T})=\frac{1}{pqr(pqr+2)}\left(
 \vcenter{\xymatrix@=2.9ex{
*{\bullet}\ar@<-.2ex>@[red]@{-}[d] \ar@<.2ex>@[green]@{-}[d] \ar@[blue]@{-}[d] & *{\bullet}\ar@<.2ex>@[red]@{-}[d] \ar@<-.2ex>@[green]@{-}[d]  \ar@[blue]@{-}[d] \\
*{\bullet} & *{\bullet}}}+
\vcenter{\xymatrix@=2.9ex{
*{\bullet}\ar@<-.2ex>@[red]@{-}[r] \ar@<.2ex>@[green]@{-}[r] \ar@[blue]@{-}[r] & *{\bullet}\\
*{\bullet}\ar@<.2ex>@[red]@{-}[r] \ar@<-.2ex>@[green]@{-}[r]  \ar@[blue]@{-}[r]  &
*{\bullet}}}+
\vcenter{\xymatrix@=2.9ex{
*{\bullet}\ar@<-.2ex>@[red]@{-}[rd] \ar@<.2ex>@[green]@{-}[rd] \ar@[blue]@{-}[rd] & *{\bullet}\ar@<.2ex>@[red]@{-}[ld] \ar@<-.2ex>@[green]@{-}[ld]  \ar@[blue]@{-}[ld] \\
*{\bullet} & *{\bullet}}}
\right),
\end{equation}
because we can view ${\mathcal T}$ as a random unit tensor in $V\otimes V\otimes V\otimes V$ and apply Proposition~\ref{integral}.

 To compute $\EE(\|{\mathcal T}\|_{\sigma,4}^4)$ we compute the inner product between (\ref{eq:TTTT}) and (\ref{eq:sigma4}).
 To perform this computation we overlay the two diagrams and count the number of cycles for each color.
 $$
 \RedFrame\cdot  \vcenter{\xymatrix@=2.9ex{
*{\bullet}\ar@<-.2ex>@[red]@{-}[d] \ar@<.2ex>@[green]@{-}[d] \ar@[blue]@{-}[d] & *{\bullet}\ar@<.2ex>@[red]@{-}[d] \ar@<-.2ex>@[green]@{-}[d]  \ar@[blue]@{-}[d] \\
*{\bullet} & *{\bullet}}}=
\vcenter{
\xymatrix@=2.9ex{
*{\bullet}\ar@[red]@{-}[d] \ar@[red]@{-}[r]  & *{\bullet}\ar@[red]@{-}[d]   \\
*{\bullet}\ar@[red]@{-}[r] & *{\bullet}}}
\ 
\vcenter{
\xymatrix@=2.9ex{
*{\bullet}\ar@[green]@<.2ex>@{-}[d] \ar@<-.2ex>@[green]@{-}[d]  &
*{\bullet}\ar@[green]@<.2ex>@{-}[d] \ar@<-.2ex>@[green]@{-}[d]  \\
*{\bullet} & *{\bullet}}}
\ 
\vcenter{
\xymatrix@=2.9ex{
*{\bullet}\ar@[blue]@<.2ex>@{-}[d] \ar@<-.2ex>@[blue]@{-}[d]  &
*{\bullet}\ar@[blue]@<.2ex>@{-}[d] \ar@<-.2ex>@[blue]@{-}[d]  \\
*{\bullet} & *{\bullet}}}=pq^2r^2.
$$
The result is
\begin{multline}
\EE(\|{\mathcal T}\|_{\sigma,4}^4)=\textstyle\frac{1}{9pqr(pqr+2)}((p^2q^2r^2+2pqr)+
2(pq^2r^2+p^2qr+pqr)+
2(p^2qr^2+pq^2r+pqr)+\\+2(p^2q^2r+pqr^2+pqr)
+2(p^2qr+pq^2r+pqr^2)=
\frac{pqr+2(pq+pr+qr)+4(p+q+r)+8}{9(pqr+2)}.
\end{multline}
A similar computation shows that
\begin{equation}
\EE\left(\|{\mathcal T}\|_{\#}^4\right)=\frac{(pq+pr+qr)+3(p+q+r)+3}{5(pqr+2)}.
\end{equation}

The following proposition shows that, in some sense,  $\|\cdot\|_\#$ is better
than $\|\cdot\|_{\sigma,4}$ as an approximation to the spectral norm $\|\cdot\|_\sigma$.
\begin{proposition}\label{prop:NormBetter}
If $p,q,r\geq 1$ then we have $\EE(\|{\mathcal T}\|_{\sigma}^4)\leq \EE(\|{\mathcal T}\|_{\#}^4)\leq \EE(\|{\mathcal T}\|_{\sigma,4}^4)$
for a random tensor ${\mathcal T}$ sampled from the uniform distribution on the unit sphere.
The inequality is strict when two of the numbers $p,q,r$ are at least 2.
\end{proposition}
\begin{proof}
We calculate
\begin{multline}
\EE(\|{\mathcal T}\|_{\sigma,4}^4)-\EE(\|{\mathcal T}\|_{\#}^4)=\\=\frac{pqr+2(pq+pr+qr)+4(p+q+r)+8}{9(pqr+2)}-\frac{(pq+pr+qr)+3(p+q+r)+3}{5(pqr+2)}=\\=\frac{5pqr+(pq+qr+rp)-7(p+q+r)+13}{45(pqr+2}=\\=
\frac{5(p-1)(q-1)(r-1)+6((p-1)(q-1)+(p-1)(r-1)+(q-1)(r-1))}{45(pqr+2)}
\end{multline}
\end{proof}
\begin{remark}
If $p=q=r=n$, then asymptotically, we have that $\EE(\|{\mathcal T}\|_{\sigma,4}^4)=O(1)$
and $\EE(\|{\mathcal T}\|_{\#}^4)=O(\frac{1}{n}).$
\end{remark}

\section{Low rank amplification}\label{sec5}
As motivation, we will first consider a map from matrices to matrices
that enhances the low rank structure.
\subsection{Matrix amplification}
Suppose $A$ is a matrix with singular values $\lambda_1\geq \lambda_2\geq \cdots\geq \lambda_r\geq 0$.  Then we can write $A=U\Sigma V^{*}$
where $U,V$ are orthogonal, $\Sigma$ is a diagonal matrix
with diagonal entries $\lambda_1,\dots,\lambda_r,$ and $V^{*}$ is the conjugate transpose of $V.$ We have
\begin{equation}
AA^{*} A=(U\Sigma V^{*})( V \Sigma^{*} U^{*})( U\Sigma V^{*})=U\Sigma^3 V^{*}.
\end{equation}
The matrix $AA^{*} A$ has singular values $\lambda_1^3\geq \lambda_2^3\geq \cdots\lambda_r^3\geq 0$.
If $\lambda_1>\lambda_2$ then the ratio of the two largest singular values increases from $\lambda_1/\lambda_2$
to $\lambda_1^3/\lambda_2^3$. If we define a map $\theta$ by \begin{equation}
\theta(A)=\frac{AA^{*} A}{\|AA^{*} A\|},
\end{equation}
where $\|\cdot\|$ is the Euclidean (Frobenius) norm, then $\lim_{n\to\infty}\theta^n(A)$ will converge to a rank $1$ matrix $B=UDV^{*}$
where 
\begin{equation}
 D=   \begin{pmatrix}
    1 & 0 & \cdots\\
    0 & 0 & \\
    \vdots & & \ddots
    \end{pmatrix}.
\end{equation}
Note that the convergence is very fast. After $n$ iterations of $\theta$,
the ratio of the two largest singular values is $\big(\frac{\lambda_1}{\lambda_2}\big)^{3^n}$.
We have that $A\cdot B=\Sigma\cdot D=\lambda_1$ is the spectral norm
of $A$ and $\lambda_1 B$ is the best rank 1 approximation of $A$ in the following sense: if $C$ is a rank 1 matrix such that $\|A-C\|$ is minimal, then $C=\lambda_1 B$.

The map $\theta$ increases the highest singular value relative to the other singular values. In this sense,
$\theta$ amplifies the sparse structure of the matrix (meaning low rank in this context).

The map $\theta$ is related to the $4$-Schatten norm, defined by
$\|A\|_{s,4}=\tr((AA^{*})^2)^{\frac{1}{4}}$. Namely, the gradient of the function
$\|A\|_{s,4}^4$ is $4AA^{*} A$ and the gradient of the function $\|A\|_{s,4}$ is equal to $AA^{*} A$ up to a scalar function.
\subsection{Tensor amplification}\label{tensor-amplification1}
We will now consider amplification of the low rank structure of tensors. 
For this we take the gradient of a spectral-like norm. Let $h=\|\cdot\|_{\#}^4$. As before, we can write
$h(\mathcal {T+E})=h_0(\mathcal {T,E})+h_1(\mathcal {T,E})+h_2(\mathcal {T,E})+h_3(\mathcal {T,E})+h_4(\mathcal {T,E})$, where $h_i$ has degree $i$ in $\mathcal{E}$ and degree $4-i$ in $T$. Now $h_0(\mathcal {T,E})=\|{\mathcal T}\|_\#^4$, the function
${\mathcal E}\mapsto h_1(\mathcal {T,E})$ is the gradient of $h$ at ${\mathcal T}$, and $h_2(\mathcal {T,E})$ is the Hessian that we have already computed.
To find a formula for the gradient $h_1(\mathcal {T,E})$ we express $h({\mathcal T})$ in diagrams
and replace each diagram by all diagrams obtained by replacing one of the closed vertices by an open vertex. Using (\ref{Tsharp}) we get
\begin{equation}
h({\mathcal T})=\|{\mathcal T}\|_\#^4=\frac{1}{5}\left(\vcenter{\xymatrix@=2.9ex{
*{\bullet}\ar@<.2ex>@[red]@{-}[r] \ar@[blue]@{-}[d] \ar@<-.2ex>@[green]@{-}[r]  & *{\bullet}\ar@[blue]@{-}[d]  \\
*{\bullet}\ar@<-.2ex>@[red]@{-}[r]  &  *{\bullet} \ar@<-.2ex>@[green]@{-}[l] }}+
 \vcenter{\xymatrix@=2.9ex{
*{\bullet}\ar@<.2ex>@[red]@{-}[r] \ar@[green]@{-}[d] \ar@<-.2ex>@[blue]@{-}[r]  & *{\bullet}\ar@[green]@{-}[d]  \\
*{\bullet}\ar@<-.2ex>@[red]@{-}[r]  &  *{\bullet} \ar@<-.2ex>@[blue]@{-}[l] }} +
 \vcenter{\xymatrix@=2.9ex{
*{\bullet}\ar@<.2ex>@[green]@{-}[r] \ar@[red]@{-}[d] \ar@<-.2ex>@[blue]@{-}[r]  & *{\bullet}\ar@[red]@{-}[d]  \\
*{\bullet}\ar@<-.2ex>@[green]@{-}[r]  &  *{\bullet} \ar@<-.2ex>@[blue]@{-}[l] }} +
2 \vcenter{\xymatrix@=2.9ex{
*{\bullet}\ar@[green]@{-}[r] \ar@[red]@{-}[d] \ar@[blue]@{-}[rd]  &  *{\bullet}  \ar@[blue]@{-}[ld] \ar@[red]@{-}[d] \\
*{\bullet}\ar@[green]@{-}[r]  &  *{\bullet} }}\right).
\end{equation}
The gradient is now equal to
\begin{equation}\label{eq:gradient}
(\nabla h)({\mathcal T})=\frac{4}{5}\left(\vcenter{\xymatrix@=2.9ex{
*{\bullet}\ar@<.2ex>@[red]@{-}[r] \ar@[blue]@{-}[d] \ar@<-.2ex>@[green]@{-}[r]  & *{\bullet}\ar@[blue]@{-}[d]  \\
*{\bullet}\ar@<-.2ex>@[red]@{-}[r]  &  *{\circ} \ar@<-.2ex>@[green]@{-}[l] }}+
 \vcenter{\xymatrix@=2.9ex{
*{\bullet}\ar@<.2ex>@[red]@{-}[r] \ar@[green]@{-}[d] \ar@<-.2ex>@[blue]@{-}[r]  & *{\bullet}\ar@[green]@{-}[d]  \\
*{\bullet}\ar@<-.2ex>@[red]@{-}[r]  &  *{\circ} \ar@<-.2ex>@[blue]@{-}[l] }} +
 \vcenter{\xymatrix@=2.9ex{
*{\bullet}\ar@<.2ex>@[green]@{-}[r] \ar@[red]@{-}[d] \ar@<-.2ex>@[blue]@{-}[r]  & *{\bullet}\ar@[red]@{-}[d]  \\
*{\bullet}\ar@<-.2ex>@[green]@{-}[r]  &  *{\circ} \ar@<-.2ex>@[blue]@{-}[l] }} +
2 \vcenter{\xymatrix@=2.9ex{
*{\bullet}\ar@[green]@{-}[r] \ar@[red]@{-}[d] \ar@[blue]@{-}[rd]  &  *{\bullet}  \ar@[blue]@{-}[ld] \ar@[red]@{-}[d] \\
*{\bullet}\ar@[green]@{-}[r]  &  *{\circ} }}\right).
\end{equation}
We can also view these diagrams with an open vertex as partial colored Brauer diagrams by removing the open vertex. For example,
\begin{equation}
\vcenter{\xymatrix@=2.9ex{
*{\bullet}\ar@[green]@{-}[r] \ar@[red]@{-}[d] \ar@[blue]@{-}[rd]  &  *{\bullet}  \ar@[blue]@{-}[ld] \ar@[red]@{-}[d] \\
*{\bullet}\ar@[green]@{-}[r]  &  *{\circ} }}\quad =\quad \vcenter{
\xymatrix@=2ex{
& *{\bullet}\ar@[green]@{-}[l]\ar@[blue]@{-}[dd]\ar@[red]@{-}[rd] & & \\
& & *{\bullet}\ar@[blue]@{-}[r] &  \\
& *{\bullet}\ar@[red]@{-}[l] \ar@[green]@{-}[ru] & & }}.
\end{equation}
Let $\Phi_\#({\mathcal T})=(\nabla h)({\mathcal T})$. We view $\Phi_\#$
as a polynomial map from $V=R\otimes G\otimes B$ to itself. This map enhances the low rank structure of a tensor ${\mathcal T}$. 

In a similar fashion, we can associate an amplification map $\Phi_{\sigma,4}$ to the norm $\|\cdot\|_{\sigma,4}$. Using (\ref{eq:sigma4}) and similar calculations as before, we get
\begin{equation}\label{eq:gradient2}
\Phi_{\sigma,4}({\mathcal T})=\frac{4}{9}\left(
\vcenter{\xymatrix@=2.9ex{
*{\bullet}\ar@<.4ex>@[red]@{-}[r] \ar@[blue]@{-}[r] \ar@<-.4ex>@[green]@{-}[r]  & *{\bullet}  \\
*{\bullet}\ar@<-.4ex>@[red]@{-}[r]\ar@[blue]@{-}[r]   &  *{\circ} \ar@<-.4ex>@[green]@{-}[l] }}+
2\vcenter{\xymatrix@=2.9ex{
*{\bullet}\ar@<.2ex>@[red]@{-}[r] \ar@[blue]@{-}[d] \ar@<-.2ex>@[green]@{-}[r]  & *{\bullet}\ar@[blue]@{-}[d] \\
*{\bullet}\ar@<-.2ex>@[red]@{-}[r]   &  *{\circ} \ar@<-.2ex>@[green]@{-}[l] }}+
 2\vcenter{\xymatrix@=2.9ex{
*{\bullet}\ar@<.2ex>@[red]@{-}[r] \ar@[green]@{-}[d] \ar@<-.2ex>@[blue]@{-}[r]  & *{\bullet}\ar@[green]@{-}[d]  \\
*{\bullet}\ar@<-.2ex>@[red]@{-}[r]  &  *{\circ} \ar@<-.2ex>@[blue]@{-}[l] }} +
2 \vcenter{\xymatrix@=2.9ex{
*{\bullet}\ar@<.2ex>@[green]@{-}[r] \ar@[red]@{-}[d] \ar@<-.2ex>@[blue]@{-}[r]  & *{\bullet}\ar@[red]@{-}[d]  \\
*{\bullet}\ar@<-.2ex>@[green]@{-}[r]  &  *{\circ} \ar@<-.2ex>@[blue]@{-}[l] }} +
2 \vcenter{\xymatrix@=2.9ex{
*{\bullet}\ar@[green]@{-}[r] \ar@[red]@{-}[d] \ar@[blue]@{-}[rd]  &  *{\bullet}  \ar@[blue]@{-}[ld] \ar@[red]@{-}[d] \\
*{\bullet}\ar@[green]@{-}[r]  &  *{\circ} }}\right).
\end{equation}

\subsection{Tensor amplification and Alternating Least Squares}\label{ALS} As we discussed in Section  \ref{notation}, 
Alternating Least Squares (ALS) is a standard approach to find low rank approximations of tensors. For rank 1, this algorithm is particularly simple. For a tensor ${\mathcal T}\in \R^{p\times q\times r}$
we try to find a rank one tensor $a\otimes b\otimes c$ such that $\|{\mathcal T}-a\otimes b\otimes c\|$
is minimal. Here $a\in \R^p$, $b\in \R^q$ and $c\in \R^r$. Unlike for higher rank, a best rank 1 approximation always exists.
The Alternating Least Squares algorithm works as follows. We start with an initial guess $a\otimes b\otimes c$. Then we fix $b$ and $c$ and update the vector $a$ such that $\|{\mathcal T}-a\otimes b\otimes c\|$ is minimal. This is a least squares regression problem that is easy to solve. 
Next, we fix $a$ and $c$ and update $b$, and then we fix $a$ and $b$ and update $c$.
We repeat the process of updating $a,b,c$ until the desired level of convergence is achieved. Numerical experiments were performed using the software {\tt MATLAB}, along with the {\tt cp\_{}als} implementation of the ALS algorithm from the package {\tt Tensor Toolbox} (\cite{toolbox}). ALS is sensitive to the choice of the initial guess. 

The default initialization for {\tt cp\_{}als} is to use a random initial guess. We will also consider
a method that we call the {\em Quick Rank 1} method. For a matrix it is easy to find the best rank 1 approximation from the singular value decomposition. If a real matrix $M$ has a singular value decomposition $M= \sum_{i=1}^{s} \lambda_i a_i b_i^{T}$ where $a_1, \ldots, a_s, b_1, \ldots, b_s$ are unit orthogonal vectors and $\lambda_1 \geq\lambda_2 \geq \ldots \lambda_s \geq 0$ are real numbers, then a best rank 1 approximation of $M$ is $\lambda_1 a_1 b_1 ^{T}$ and $M \cdot a_1 b_1^{T}=\lambda_1$ is the spectral norm of $M$. (The best rank 1 approximation is unique when $\lambda_1>\lambda_2$.)

For the {\em Quick Rank 1} method, we use the matrix case to find an initial rank 1 approximation of a given tensor ${\mathcal T}$
of size $p\times q\times r$. We  flatten (unfold) this tensor to a $p \times qr$ matrix ${\mathcal T}_{(1)}$. Then we find a rank 1 approximation of this matrix as described above. Let ${\mathcal T}_{(1)} \approx \beta a d^{T}~\text{where}~a,d~\text{are orthogonal vectors}$ and $\beta \in \mathbb{R}.$ We convert the vector $d$ of dimension $qr$ to a $q \times r$ matrix $D$. Now we find the best rank 1 approximation of $D \approx \gamma b c^{T}$ such that $b,c$ are orthogonal vectors and $\gamma \in \mathbb{R}.$ We will use $(\beta\gamma) a\otimes b\otimes c$ as a rank 1 approximation
to the tensor~${\mathcal T}$.

Tensor amplification can be used to obtain better initial guesses for ALS, so that better rank 1 approximations can be found using fewer iterations in the ALS algorithm.
We will consider 4 different ways of choosing an initial guess for ALS:
\begin{enumerate}    \item {\bf Random.} We choose a random initial guess for the rank 1 approximation. 
    \item {\bf Quick Rank 1.} We first use the quick rank 1 method described above.
    \item {\bf $\Phi_{\sigma,4}$ and Quick Rank 1.} We apply the Quick Rank 1 method to $\Phi_{\sigma,4}({\mathcal T})$.
    \item {\bf $\Phi_{\#}$ and Quick Rank 1.} We apply the Quick Rank 1 method to $\Phi_{\#}({\mathcal T})$.
\end{enumerate}

Rank 1 approximation methods given above can be generalized to higher ranks. Low rank tensor approximation problem is given in \eqref{decomp} and \eqref{rankr}. Let $\mathcal{T} \in \mathbb{R}^{p_1 \times p_2 \times  p_3}$ be a tensor of order 3.  We will look for a rank $r \geq 2$ approximation $\mathcal{S}$ such that 
\begin{equation}
\|\mathcal{T-S}\|~\text{is minimal with}~\mathcal{S}=\llbracket \Lambda~;~U^{(1)}, U^{(2)}, U^{(3)}\rrbracket
\end{equation}
where the factor matrices~$U^{(i)} \in \mathbb{R}^{p_i \times r}~\text{for}~1\leq i \leq 3$ and $\Lambda \in \mathbb{R}^{r}.$ 

ALS method starts with a random initial guess for the factor matrices. We first fix  $U^{(2)}$ and $U^{(3)}$ to solve for $U^{(1)},$ then fix $U^{(1)}$ and $U^{(3)}$  to solve for $U^{(2)},$ and then fix $U^{(1)}$ and $U^{(2)}$ to solve for $U^{(3)}.$ This iterative process continues until some convergence criterion is satisfied.

For the \textit{iterative Quick Rank 1 method}, we first employ the Quick Rank 1 method to approximate $\mathcal{T}$ with a rank 1 tensor $\lambda_1 a_1 \otimes b_1 \otimes c_1.$  The process continues iteratively; and at each step Quick Rank 1 method is used to find a rank 1 approximation of $\mathcal{T} - \sum_{i=1}^{s} \lambda_i a_i \otimes b_i \otimes c_i$ for $2 \leq s \leq r-1.$ 

As in the rank 1 case, we use 4 different methods to choose an initial guess for the ALS method for the low rank $r$ decomposition of $\mathcal{T}:$
\begin{enumerate}    \item {\bf Random.} We choose a random initial guess for the factor matrices.
    \item {\bf Quick Rank 1.} We use an iterative approach based on Quick Rank 1 method  as described above. (Algorithm 5.1, $k=0.$)
    \item {\bf $\Phi_{\sigma,4}$ and Quick Rank 1.} We iteratively apply the Quick Rank 1 method to $\Phi_{\sigma,4}({\mathcal T}).$ (Algorithm 5.1,~$k=1.$)
    \item {\bf $\Phi_{\#}$ and Quick Rank 1.} We iteratively apply the Quick Rank 1 method to $\Phi_{\#}({\mathcal T}).$ (Algorithm 5.1,~$k=2.$)
\end{enumerate}

\begin{algorithm}[h!]
\begin{algorithmic}[1]
\Function{${\rm \textcolor{blue}{rank\_r\_methods}}$}{$\mathcal{T},r,k$}
\State $\mathcal{D}=\mathcal{T}$
\State $s=0$
\While{$ s < r$}
 \State $s\leftarrow s+1$
 \If {$k=0$}
  \State $\mathcal{U} \leftarrow \mathcal{D}$
 \ElsIf {$k=1$}
  \State $\mathcal{U} \leftarrow \Phi_{\sigma,4}(\mathcal{D})$
 \Else
    \State $\mathcal{U} \leftarrow \Phi_{\#}(\mathcal{D})$
 \EndIf
\State Approximate $\mathcal{U}$ with a unit rank 1 tensor via Quick rank 1 method:  $\mathcal{U} \approx \lambda_s v_s = \lambda_s a_s \otimes b_s\otimes c_s$
\State Update the coefficients $\lambda_1,\dots,\lambda_s$ such that $\|\mathcal{D}\|$ is minimal, where $\mathcal{D}=\mathcal{T} -
\sum_{i=1}^{s} \lambda_i v_i$
\EndWhile
\State {\bf return} decomposition $\mathcal{T} = \mathcal{S} + \mathcal{D}$ where $\mathcal{S}=\sum_{i=1}^r \lambda_i v_i$ \\
\EndFunction
\end{algorithmic}
\caption{Low Rank approximation to tensor $\mathcal{T}$ of order 3}\label{alg1}
\end{algorithm}

\section{Experiments}\label{experiment}
\subsection{Rank 1 approximation}\label{rank1test}
In our experiments, we started with a random $30 \times 30 \times 30$ unit tensor of rank 1,  $ {\mathcal T}=a\otimes b\otimes c,$ where $a,b,c  \in \mathbb{R}^{30}$ are random unit vectors, independently drawn from the uniform distribution on the unit sphere. We then added  a random tensor ${\mathcal E}$ of size $30 \times 30 \times 30$ with $\|{\mathcal E}\|=10$ to obtain
a noisy tensor ${\mathcal T}_n=\mathcal{T+E}$. The noise tensor ${\mathcal E}$ is
chosen from the sphere in $\R^{30\times 30\times 30}\cong \R^{27000}$ of radius 10 with uniform distribution. Note that there is more noise than the original signal. The signal to noise ratio is $20\log_{10}(1/10)=-20$\;dB. We used four methods for rank 1 approximation. Each method gives a rank 1 tensor $\lambda a'\otimes b'\otimes c'$. To measure how good the rank 1 approximation is to the original tensor ${\mathcal T}$, we compute the inner product 
$${\mathcal T}\cdot (a'\otimes b'\otimes c')=
(a\otimes b\otimes c)\cdot (a'\otimes b'\otimes c')=(a\cdot a')(b\cdot b')(c\cdot c'),
$$
which we will call the fit. The fit is a number between $0$ and $1$ where $1$ means a perfect fit.

We created 1000 noisy tensors of size $30 \times 30 \times 30$ as described above. We
ran each of the 4 methods to find the best rank 1 approximation for each of the 1000 tensors. 
For the random initial guess method, we repeated the calculation 10 times
with different random initial guesses and recorded the best fit, total number of ALS iterations,
and total running time. All other methods were only run once and the fit, total number of ALS iterations, and running time were calculated. For all records, we took the average and standard deviation. 

There is a tolerance parameter $\varepsilon$ in the ALS implementation in {\tt Tensor Toolbox}.
The algorithm terminates if the fit after an iteration increases by a factor smaller than $1+\varepsilon$. For the default value $\varepsilon=10^{-4}$ we obtained the following results:

\begin{table}[h!]
\label{neg4}
\centering
\begin{tabular}{|c|c|c|c|}
\hline
   \textbf{Random (10 runs)}  & Max Fit  & Total \# Iterations & Total Time\\
 
    Average &  0.7136 & 77.5080 & 0.0943 \\
    
   Standard Deviation  &  0.2715& 12.0254 & 0.0159\\
   \hline
\hline
    \textbf{Quick Rank 1} & Fit & \# Iterations & Time\\
   
    Average &  0.7848 & 2.94 & 0.0177 \\
    
   Standard Deviation  &  0.1618 & 1.2345 & 0.0025\\
   \hline

  \textbf{$\Phi_{\sigma,4}$ and Quick Rank 1}  & Fit  & \# Iterations & Time\\
  
    Average &  0.8010 & 2 & 0.0210\\
     
   Standard Deviation  &  0.1256 & 0 & 0.0027\\
   \hline

   \textbf{$\Phi_{\#}$ and Quick Rank 1} & Fit  & \# Iterations & Time\\
   
    Average &  0.8178 & 2 & 0.0205 \\
    
   Standard Deviation  &  0.0515 & 0 & 0.0025\\
   \hline

\end{tabular}
\caption{A comparison of rank-1 approximation methods with tolerance parameter $\varepsilon=10^{-4}$}
\end{table}
It can be observed from Table \ref{neg4} that a better fit is obtained by using tensor amplification rather than a random initial guess. Even if we take the best case of repeating ALS for 10 different random initial conditions, quick rank with amplification still yields a better fit. The total number of ALS iterations with random initial guess is much larger than for the quick rank 1 initialization, or quick rank 1 with tensor amplification. On average, the number of iterations for the best run with random initialization is 10.44, which is much larger than the number of iterations after tensor amplification, which is 2. The running time 
is also favorable for the quick rank 1 initialization.
Amplification gives a better fit for the rank 1 approximation, while the running time has only marginally increased.

If we change the tolerance to $\varepsilon=10^{-6}$ then the number of iterations increases and the results are given in Table \ref{neg6}. As shown in the table, the amplification $\Phi_\#$ performs better than the amplification $\Phi_{\sigma,4}$. This is expected, as the norm $\|\cdot\|_\#$ is a better approximation for the spectral norm than $\Phi_{\sigma,4}$. We see that the amplification $\Phi_{\#}$ combined with the quick rank 1 method still yields a better fit than the best-out-of-10 runs
with random initialization. The number of iterations for the random initialization approximation with the best fit is 25.78 on average, while the average number of ALS iterations for $\Phi_{\#}$ combined with quick rank 1 is only $3.54$.

\begin{table}[h!]
\label{neg6}
\centering
\begin{tabular}{|c|c|c|c|}
\hline
   \textbf{Random (10 runs)}  & Max Fit  & Total \# Iterations & Total Time\\

    Average &  0.8120& 290.3230 & 0.2893 \\
     
   Standard Deviation  &  0.0914 & 82.7586 & 0.0803\\
   \hline
\hline
   \textbf{Quick rank 1}  & Fit  & \# Iterations & Time\\
 
    Average &  0.7955 & 6.9780&0.0210 \\
     
   Standard Deviation  &  0.1436 & 4.6320 & 0.0048\\
   \hline

  \textbf{$\Phi_{\sigma,4}$ and Quick Rank 1}  & Fit  & \# Iterations & Time\\
   
    Average &  0.8091 & 2.18 & 0.0238\\
     
   Standard Deviation  &  0.0999 & 1.2603 & 0.0046\\
   \hline

   \textbf{$\Phi_{\#}$ and Quick Rank 1} & Fit  & \# Iterations & Time\\
   
    Average &  0.8180 & 3.54 & 0.0234\\
   
   Standard Deviation  &  0.0511 &0.69 & 0.0029\\
   \hline

\end{tabular}
\caption{A comparison of rank-1 approximation methods with tolerance parameter $\varepsilon=10^{-6}$}
\end{table}

\subsection{Rank 2 approximation}\label{rank2test}

We started with a random $40 \times 40 \times 40$ unit tensor of rank 2,  $ {\mathcal T}=a_1\otimes b_1\otimes c_1 + a_2 \otimes b_2 \otimes c_2$, where $a_1,b_1,c_1,a_2,b_2,c_2  \in \mathbb{R}^{40}$ are random unit vectors, independently drawn from the uniform distribution on the unit sphere. We then added  a random tensor ${\mathcal E}$ of size $40 \times 40 \times 40$ with $\|{\mathcal E}\|=10$ to obtain
a noisy tensor ${\mathcal T}_n=\mathcal{T+E}$. The noise tensor ${\mathcal E}$ is
chosen from the sphere in $\R^{40\times 40\times 40}\cong \R^{64000}$ of radius 10 with uniform distribution. Each method gives a rank 2 tensor ${\mathcal S}$ of size $40\times 40\times 40$ and the  fit of the approximation is given by  $({\mathcal T}\cdot \mathcal{S})/ \| \mathcal{S} \|.$ As in Section \ref{rank1test}, we created 1000 noisy tensors of size $40 \times 40 \times 40$ and we ran each of the 4 methods to find a best rank 2 approximation for each tensor. Random initial guess method is repeated 10 times for each tensor and the best fits, total number of iterations and total running times were recorded. The other three methods were run only once and the fit, total number of ALS iterations, and running time were recorded for each tensor. For the tolerance parameter $\varepsilon=10^{-4},$  the average and the standard deviation of all the records are given in Table \ref{rank2}.

\begin{table}[h!] \label{rank2}
\centering
\begin{tabular}{|c|c|c|c|}
\hline
   \textbf{Random (10 runs)}  & Max Fit  & Total \# Iterations & Total Time\\
 
    Average & 0.6665 & 92.2550 & 0.1195 \\
    
   Standard Deviation  &  0.2411& 11.4910 & 0.0138\\
   \hline
\hline
    \textbf{Quick Rank 1} & Fit & \# Iterations & Time\\
   
    Average &  0.6788 & 2.1760 & 0.0925 \\
    
   Standard Deviation  &  0.1700 & 0.8425 & 0.0114\\
   \hline

  \textbf{$\Phi_{\sigma,4}$ and Quick Rank 1}  & Fit  & \# Iterations & Time\\
  
    Average &  0.7040 & 2.0790 & 0.0989\\
     
   Standard Deviation  &  0.1579 & 0.5244 & 0.0115\\
   \hline

   \textbf{$\Phi_{\#}$ and Quick Rank 1} & Fit  & \# Iterations & Time\\
   
    Average &  0.7607 & 2.0450 & 0.0989 \\
    
   Standard Deviation  &  0.1079 & 0 .3809& 0.0117\\
   \hline
\end{tabular}
\caption{A comparison of rank 2 approximation methods with tolerance parameter $\varepsilon=10^{-4}$}
\end{table}

\section{Conclusion}
Colored Brauer diagrams are a graphical way to represent invariant features in tensor data and can be used to visualize calculations with higher order tensors, and analyse the computational complexity of related algorithms. 
We have used such graphical calculations to find approximations of the spectral norm and to define polynomial maps that amplify the low rank structure of tensors. Such amplification maps are useful for finding better low rank approximations of tensors and are worthy of further study. We are interested in studying $n$-edge-colored large Brauer diagrams when $n > 3$ and generalizing the given methods for the tensors of order greater than 3.     
The complexity of computing invariant features corresponding to large diagrams can be high, depending on the particular diagram. In future research, we will investigate how one can improve such computations by using low rank tensor approximations for intermediate results within the calculations. 
\section{Acknowledgements}
This work was partially supported by the National Science
Foundation under Grant No. 1837985 and by the Department of Defense under Grant No.\\BA150235. Neriman Tokcan was partially supported by University of Michigan Precision Health Scholars Grant No. U063159.

\end{document}